\documentclass[12pt]{amsart}
\usepackage{enumerate}
\usepackage{hyperref}
 \theoremstyle{plain}
 \newtheorem{theorem}{Theorem}

 \newtheorem{lemma}{Lemma}
 \newtheorem{corollary}{Corollary}
\newtheorem{proposition}{Proposition}

\theoremstyle{definition}
\newtheorem{definition}{Definition}
\newtheorem{example}{Example}

 \newtheorem{remark}{Remark}

\newcommand{\G}{\mathbb G}
 \newcommand{\N}{\mathbb N}
 
 \newcommand{\Z}{\mathbb Z}
 \newcommand{\R}{\mathbb R}
 \newcommand{\Q}{\mathbb Q}

\newcommand{\cA}{\mathcal A}
 \newcommand{\cC}{\mathcal C}
 \newcommand{\cD}{\mathcal D}
 \newcommand{\cK}{\mathcal K}
 \newcommand{\cW}{W^{1,1}}
 \newcommand{\cWD}{\cW_\cD}
 \newcommand{\cVD}{\cV_\cD}
 \newcommand{\cF}{\mathcal F}
\newcommand{\cG}{\mathcal G}
\newcommand{\cH}{\mathcal H}
  \newcommand{\cL}{\mathcal L}

 \newcommand{\cM}{\mathcal M}
 
\newcommand{\cP}{\mathcal P}

\newcommand{\cV}{\mathcal V}

\newcommand{\ovX}{\overline{X}}
\newcommand{\unX}{\underline{X}}

\newcommand{\bfh}{\mathbf{h}} 
\newcommand{\fM}{\mathfrak M}

 \newcommand{\bH}{\bar{H}}

\newcommand{\bL}{\bar{L}}
\newcommand{\bM}{\bar{M}}

 \newcommand{\af}{\alpha}
 \newcommand{\fui}{\varphi}
 \newcommand{\ep}{\varepsilon}
\newcommand{\be}{\beta}
\newcommand{\ga}{\gamma}
\newcommand{\bga}{\bar\gamma}
 \newcommand{\ka}{\kappa}
 
 \newcommand{\de}{\delta}
 \newcommand{\De}{\Delta}
 \newcommand{\om}{\omega}
 \newcommand{\Om}{\Omega}
  \newcommand{\lam}{\lambda}

 \newcommand{\te}{\theta}
 
 \newcommand{\dga}{\dot\gamma}
 \newcommand{\lV}{\left\Vert}
 \newcommand{\rV}{\right\Vert}
 \newcommand{\rip}{\rangle}
 \newcommand{\lip}{\langle}
 \newcommand{\mL}{\mathbb L}
 \newcommand{\mH}{\mathbb H}

\newcommand{\diver}{\operatorname{div}}

\newcommand{\Lip}{\operatorname{Lip}}

\newcommand{\diam}{\operatorname{diam}}

\newcommand{\entre}{\setminus}

\newcommand{\sop}{\operatorname{supp}}

\begin{document}
\title{Weak KAM theory for subriemannian Lagrangians}
\author{H\'ector S\'anchez Morgado}
\address{Instituto de Matem\'aticas, UNAM. Ciudad Universitaria C. P. 
          04510, Cd. de M\'exico, M\'exico.}
 \email{hector@math.unam.mx}
\begin{abstract}
We extend weak KAM theory to Lagrangians that are defined
  only on the horizontal distribution of a sub-Riemannian
  manifold. The main tool is Tonelli's theorem which allows dispending
  on a Lagrangian dynamics.
  \end{abstract}
\maketitle

\section{Introduction}
\label{sec:introduction}

Lagrangians that are not defined on the whole tangent bundle of a
manifold arise naturally in mechanics with non-holonomic constraints.
We approach the study of such kind of systems based on variational
principles, more specifically the study of global minimizing curves
and the related weak KAM solutions.

A basic fact that should be stablished to start, is a Tonelli type theorem.
We accomplish that in section \ref{sec:tonelli-thm}.
In section \ref{sec:disc-value} we consider the infinite horizon
discounted problem and obtain some bounds for the discounted value
function, uniform respect to the  discount.
Using those bounds we prove in section \ref{sec:wkam} the existence of
weak KAM solutions. Then we introduce the Lax-Oleinik semigroup, the
action potential, the Peierls barrier and the Aubry set, and stablish
usual properties of these objects. We also prove the long time convergence of
the Lax-Oleinik semigroup under a suitable growth condition for the Lagrangian.
In section \ref{sec:HJ} we consider viscosity solutions of the
Hamilton-Jacobi equation and prove that weak KAM solutions are
viscosity solutions. We don't know if the converse is true.
In section \ref{sec:aubry-mather} we considere minimizing measures and
define the effective Lagrangian and Hamiltonian. 
We don't kow if the support of any minimizing measure is
contained in the Aubry set.
In section \ref{sec:conv-disc-value} we observe that the normalized
discounted value function converges, as in \cite{DFIZ}, to a particular
weak KAM solution.
In section \ref{sec:homog} we consider homogenization of the Hamilton
Jacobi following the approach in \cite{CIS}.

Several authors have studied some aspects of Calculus of Variations
and Hamilto- Jacobi equation in our setting.  

Gomes \cite{G} has studied the existence and properties of minimizing measures
as well as the solutions of the dual min-max problem. 

Manfredi and Stroffolini \cite{MS} gave a Hopf-Lax formula in the Heisenberg group,
and Balogh et al \cite{BCP} did it for Carnot groups.

Birindellli and Wigniolli \cite{BW} studied homogenization of the
Hamilton-Jacobi equation in the Heisenberg group and Stroffolini
\cite{S} did it for Carnot groups.
\section{Tonelli's theorem}
\label{sec:tonelli-thm}

\subsection{Preliminaries}
\label{sec:preliminaries}

Let $(M, \cD, \lip,\rip)$ be a sub-Riemannian manifold such that
$\cD$ is bracket generating. Denote by $\pi:TM\to M$ the natural projection.
\begin{definition}We say that
  \begin{itemize}
  \item 
 A continuous curve $\ga : [a, b]\to M$ is {\em absolutely continuous} iff
$\fui\circ\ga$ is  absolutely continuous for any smooth and compactly
supported $\fui:M\to\R$.
 
\item 
An absolutely continuous curve $\ga:[a,b]\to M$ is {\em horizontal} if
$\dga(t)\in\cD$  for a. e. $t\in[a,b]$.
\item
  We denote by $\cWD([a,b])$
the set of horizontal absolutely continuous curves defined on the
interval $[a,b]$. 
\end{itemize}
\end{definition}
  
  \begin{proposition}
    A continuous curve $\ga : [a, b]\to M$ is absolutely continuous if and only if
    there is a partition $a = t_0 < t_1 <\cdots < t_p = b$ and charts $(U_i,\fui_i)$ of
   $M$ for $1\le i\le p$ such that $\ga([t_{i-1}, t_i])\subset U_i$,
   and $\fui_i\circ\ga|[t_{i-1},t_i]$ is absolutely continuous.
  \end{proposition}
 
For $\ga\in\cWD([a,b])$ we define its
sub-Riemannian length by
\[\ell(\ga)=\int_a^b\|\dga(t))\|\ dt.\]
From that we define the {\em Carnot-Catatheodory} distance $d$ on $M$ which
in turn defines a topology on $M$ that coincides with the
original topology of $M$ as a manifold.
\begin{definition}
  \begin{itemize} 
   \item  A family $\cF$  of curves $\ga:[a,b]\to M$ is $d$-absolutely
  equicontinuous if $\forall\ep>0$, $\exists\de>0$
  s.t. $\forall\ga\in\cF$ and any  disjoint subintervals $ ]a_1,b_1[,\cdots,]a_N,b_N[$ 
 of $[a,b]$:
  \[\sum_{i=1}^Nb_i-a_i<\de\implies \sum_{i=1}^Nd(\ga(b_i),\ga(a_i))<\ep\]
\item A
 family $\cG\subset L^1([a,b],m)$,  $m$ the Lebesgue measure on
 $[a,b]$, is uniformly integrable 
if given $\ep>0$ there is $\de>0$ such that 
\[f\in\cG, E\subset[a,b]\hbox{ measurable}, m(E)<\de\implies
  \int_E|f|<\ep.\]
   \end{itemize}
 
\end{definition}
\begin{remark}\label{basic}\quad
  \begin{enumerate}
  \item  A family $\cF\subset\cWD([a,b])$ such that $\{\|\dga\| : \ga\in\cF\}$
    is uniformly integrable, is $d$-absolutely equicontinuous.
  \item A $d$-absolutely equicontinuous family is $d$-equicontinuous.
  \item An uniform limit of $d$-absolutely equicontinuous  functions is $d$-absolutely continuous.
   %\item The sub-Riemannian length of a $d$-absolutely equicontinuous family
     %$\cF\subset \cWD([a,b])$ is uniformly bounded, i.e. there is $L>0$ s.t.              
 % $\ell(\ga)\le L$ $\forall\ga\in\cF.$
\end{enumerate}
\end{remark}
\begin{theorem}\label{hor}
  Let $(\ga_n)_n$ be a sequence in $\cWD([a,b])$ that converges $d$-uniformly to
  $\ga$, and such that
  $( \|\dga _n\|)_n$ is uniformly integrable.   Then $\ga\in\cWD([a,b])$.
  \end{theorem}
  \begin{proof}
   Step 1. Denote by $V(\eta)$ the total $d$-variation of a $d$-absolutely
  continuous curve $\eta$.
  Then,  the function $v:[a,b]\to\R$ defined by
  $v(t)=V(\ga|[a,t])$ is absolutely
  continuous. We observe that according to excersice 3.52 in
  \cite{ABB}, $V(\eta)=\ell(\eta)$.

  Let $[c,d]\subset[a,b]$, and $c=t_0<t_1<\cdots<t_m=d$, then
  \[\sum_{i=1}^md(\ga(t_{i-1}),\ga(t_i))=\lim_{n\to\infty}\sum_{i=1}^md(\ga_n(t_{i-1}),\ga_n(t_i))  \le  \liminf_{n\to\infty}V(\ga_n|[c,d])\]
  Thus $V(\ga|[c,d])  \le  \liminf\limits_{n\to\infty}V(\ga_n|[c,d])$.
  
  Given $\ep>0$ there is $\de>0$ such $n\in\N,   m(E)<\de\implies\int_E  \|\dga _n\|<\ep$
 . Let $ ]c_1,d_1[,\cdots,]c_k,d_k[$
  be disjoint subintervals of $[a,b]$,  with $\sum\limits_{i=1}^k(d_i-c_i)<\de$, then
\begin{align*}
  0\le\sum_{i=1}^k v(d_i)-v(c_i)  &=\sum_{i=1}^kV(\ga|[c_i,d_i]) \le  \sum_{i=1}^k \liminf_{n\to\infty}V(\ga_n|[c_i,d_i]) \\&
\le \liminf_{n\to\infty}\sum_{i=1}^kV(\ga_n|[c_i,d_i])=
  \liminf_{n\to\infty}\sum_{i=1}^k \int_{c_i}^{d_i}\|\dga _n\|\le\ep
  \end{align*}
Step 2. There is reparametrization $\bga:[0, L]\to M$ of $\ga$,
    $L=V(\ga)$, such that 
$d(\bga(r),\bga(s))\le |r-s|$ for all $r,s\in[0,L]$,
and so $\ga\in\cWD([a,b])$.

Since $v$ is nondecreasing and absolutely continuous, the function
$\phi:[0,L]\to[a,b]$ defined by
\[\phi(s)=\inf\{t: v(t)\ge s\}\]
satisfies $v(\phi(s))=s$.
$\phi$ is strictly increasing with a denumerable set of jump discontinuities
\[\lim_{s\to r_n-} \phi(s)=a_n < b_n =\lim_{s\to r_n+}\phi(s),\]
which correspond to the intervals $[a_n,b_n]$ where $v$ is constant.
 $v$ is a bijection from $[0,1]-\bigcup[a_n,b_n[$ onto $[0,L]$.

For a Riemannian metric $g$ on $M$ defining a distance  $d_g$ there is $C>0$ such that
\[Cd_g(\ga(t),\ga(s))\le
d(\ga(t),\ga(s))\le V(\ga|[t,s])=v(s)-v(t),\]
thus, if $v'(t)=0$ then $\dga(t)=0$. If $t\in[a,b]$ is a number where $v$
and $\ga$ are differentiable and $\dga(t)\ne 0$, then $t\notin\bigcup
[a_n,b_n]$, and thus $t=\phi(s)$ for a number $s\in[0,L]$ where $\phi$
is continuous. Then
\[\lim_{r\to s}\frac{\phi(r)-\phi(s)}{r-s}=
  \lim_{r\to s}\frac 1{\dfrac {v(\phi(r)) -v(\phi(s))} {\phi(r)-\phi(s)}} =\frac 1{v'(t)},\]
so that $\phi$ is differentiable at $s$.

Define $\bga:[0, L]\to M$ by $\bga(s)=\ga(\phi(s))$, for $r < s$
we have
\[
d(\bga(r),\bga(s))=d(\ga(\phi(r),\ga(\phi(s)))\le V(\ga|[\phi(r),\phi(s)]) =v(\phi(s))-v(\phi(r))=s-r
\]
By Proposition 3.50 in \cite{ABB} we have that $\bga\in\cWD([0,L])$.
If $t\in[a,b]$ is a number where $v$ and $\ga$ are differentiable and
$\dga(t)\ne 0$, then $t=\phi(s)$ for a number $s\in[0,L]$ where $\phi$ is
differentiable, thus so is $\bga$ and $
\dot{\bga}(s)=\dga(t)/v'(t) $.
Thus $\ga\in\cWD([a,b])$.

\end{proof}
\subsection{Tonelli's Theorem}
\label{sec:tonelli}
We assume that $L:\cD\to\R$ is a  $C^2$ function satisfying the following properties
\begin{description}
\item [Uniform superlinearity] For all $K\ge 0$ there is $C(K)\in\R$.
  such that \[L(v)\ge K\|v\|+C(K)\hbox{ for all }v\in\cD.\]
\item [Uniform boundedness] For all $R\ge 0$, we have
  \[A(R)=\sup\{L(v):\|v\|\le R\}<+\infty.\]
\item [Strict convexity] there is  $a>0$ such that $\partial^2_ {vv} L (v)
  (w,w)\ge a$  for all  $ v,w\in\cD$ with $\|w\|=1$,
where $\partial_v L$ denotes the derivative along the fibers.
\end{description}

We define the {\em Energy} function $E:\cD\to\R$ by
$E(v)=\partial_v L(v)\cdot v-L(v)$. Notice that
$\dfrac{dE(sv)}{ds}=s\partial^2_ {vv} L(sv)(v,v)$,
and therefore
\[E(v)=E(0)+\int_0^1 \frac{dE(sv)}{ds}\ ds=-L(0)+\int_0^1
  s\partial^2_ {vv} L(sv)(v,v)\ge -A(0)+a\|v\|^2.\]
For $\ga\in\cWD([a,b])$, $\lam\ge 0$ we define the discounted action
\[A_{L,\lam}(\ga)=\int_a^be^{t\lam}L(\dga(t))\ dt.\]
\begin{lemma}\label{UI}
  For $c\in\R$ let
  $\cF(c,r)=\{\ga\in\cWD([a,b]): A_{L,\lam}(\ga)\le c, \lam\in[0,r]\}. $Then the families
  $\{e^{\lam t}\|\dga(t)\|:\ga\in\cF(c,r),\lam\in[0,r]\}$, $\{\|\dga(t)\|:\ga\in\cF(c,r)\}$ are uniformly integrable.
\end{lemma}
\begin{proof}
  By the uniform superlinearity,  $L(v)\ge C(0)$ .
    Given $\ep>0$ take $K>0$ such that 
  $$
  \max_{\lam\in[0,r]}\Big(c-C(0) \int_a^be^{\lam t}dt\Big)<K\ep.
  $$
  By the uniform superlinearity,
  \[L(v)\ge K\|v\|+C(K).\]
  Let $\ga\in\cF(c)$ and $E\subset[a,b]$ be measurable, then  
\[C(K)\int_E e^{\lam t}dt+K\int_Ee^{\lam t}\|\dga(t)\|\ dt
   \le\int_E e^{\lam t} L(\dga(t))\ dt\]
and
   \[C(0)\int_{[a,b]\entre E}e^{\lam t}dt\le \int_{[a,b]\entre E} e^{\lam
       t} L(\dga(t))\ dt \]
  Adding the inequalities we get 
  \[(C(K)-C(0))\int_Ee^{\lam t}dt+C(0)\int_a^be^{\lam t}dt+K\int_E
    e^{\lam t}\|\dga(t)\|\ dt    \le A_{L,\lam}(\ga)\le c\]
which gives
\[\int_Ee^{\lam t}\|\dga(t)\|\ dt\le\frac{c-C(0)\int_a^be^{\lam t}dt}K+
  \frac{C(0)-C(K)}K\int_Ee^{\lam t}dt
  <\ep+\frac{C(0)-C(K)}{K}\int_E e^{\lam t}dt.\]
This gives the uniform integrability of $\{e^{\lam t}\|\dga(t)\|:\ga\in\cF(c,r), \lam\in[0,r]\}$.
The uniform integrability of $\{\|\dga(t)\|:\ga\in\cF(c,r)\}$ follows from
\[\int_Ee^{\lam t}\|\dga(t)\| dt\ge e^{\lam a}\int_E\|\dga\|\ge \min_{\lam\in[0,r]}e^{\lam a}\int_E\|\dga\|\]
\end{proof}
  
\begin{theorem}\label{Cle}
   Let $(\ga_n)_n$ be a sequence in $\cWD([a,b])$ converging $d$-uniformly to
   $\ga$, $\lam_n\ge 0$  converging to $\lam$  and
   \[\liminf_{n\to\infty}A_{L,\lam_n}(\ga_n)<+\infty.\]
   Then $\ga\in\cWD([a,b])$ and
   \begin{equation}
     \label{eq:lsc}
     A_{L,\lam}(\ga)\le \liminf_{n\to\infty}A_{L,\lam_n}(\ga_n).
   \end{equation}
\end{theorem}
\begin{proof}
  
Let $l=\liminf\limits_{n\to\infty}A_{L,\lam_n}(\ga_n)$,  extracting subsequencse
still denoted $\ga_n$, $\lam_n$  we have  $l=\lim\limits_{n\to\infty}A_{L,\lam_n}(\ga_n)$, and forgetting
some of the first curves $\ga_n$, 
we can suppose that $\ga_n\in\cF(l+1,\lam+1)$ for all $n$. 
Lemma \ref{UI} implies that $(\|\dga_n\|)_n$ and $(e^{\lam_n t}\|\dga_n(t)\|)_n$
are uniformly integrable, and Theorem \ref{hor} that $\ga\in\cWD([a,b])$.
  
Let us show how we can reduce the proof of \eqref{eq:lsc} to the case where 
$M$ is an open subset of $\R^d$,  $d =\dim M$ and the horizontal 
distribution is trivial. The lagrangian $L$ is bounded
below by $C(0)$ . If $[a',b']\subset[a,b]$, for all $n$ we have

\[A_{L,\lam_n}(\ga_n|[a',b'])\le
  A_{L,\lam_n}(\ga_n)-C(0)\int_{[a,b]\entre[a',b']}e^{\lam_n t}dt\]
so that
\[\liminf_{n\to\infty}A_{L,\lam_n}(\ga_n|[a',b'])<+\infty.\]
Let now the partition $a_0 = 0 < a_1 < \ldots < a_p = 1$ and
$U_1, \ldots, U_p$ be as in the proof of Theorem \ref{hor}, so that
the horizontal distribution is trivial on $U_i$  and
$\ga([a_{i-1},a_i])\subset U_i$, $i=1,\ldots,p$.
It is enough to prove that  
\[A_{L,\lam_n}(\ga|[a_{i-1},a_i])\le \liminf_{n\to\infty}A_{L,\lam_n}(\ga_n|[a_{i-1},a_i])\]
because that implies 
\begin{align*}
  A_{L,\lam}(\ga) =\sum_{i=1}^pA_{L,\lam}(\ga|[a_{i-1},a_i]) &\le
\sum_{i=1}^p\liminf _{n\to\infty} A_{L,\lam_n}(\ga_n|[a_{i-1},a_i])\\
&\le \liminf_{n\to\infty}\sum_{i=1}^pA_{L,\lam_n}(\ga_n|[a_{i-1},a_i])=\liminf_{n\to\infty}A_{L,\lam}(\ga_n)
\end{align*}
In the sequel we suppose that $M=U$ is an open set of $\R^d$, the
horizontal distribution is $U\times\R^m$ with $\langle ,\rangle$ the
euclidean inner product in $\R^m$, and that $\ga([a,b])$ and 
all $\ga_n([a,b])$ are contained $U$.
We will write $\dga(t)=(\ga(t), \zeta(t))$, $\dga_n(t)=(\ga_n(t), \zeta_n(t))$.

   \begin{lemma}\label{LemCle}
Let $U\subset\R^d$ be open, $L\in C^2(U\times\R^m)$ be stricly convex
and uniformly superlinear on the second variable with
$\sup\{L(x,v):\|v\|\le R\}<+\infty$ for each $R>0$.
     Given $K\subset U$ compact, $C>0$ and $\ep>0$, there exists $\de>0$
  such that if $x\in K$, $| x-y|\le\de$, $\|v\|\le C$ and $w\in\R^m$,
  then
  \[%begin{equation}\label{Cle}
  L(y,w)\ge L(x,v)+\partial_v L(x,v)\cdot(w-v) -\ep.
  \]%end{equation}
  \end{lemma}
  For a fixed constant $C$ let
   \[E_C=\{t\in[a,b]:\|\zeta(t)\|\le C\}  \]
  Given $\ep> 0$, the compact set $K=\ga([a, b])\cup\bigcup_{n\in\N} \ga_n[a, b]$ 
and the  constant $C$ fixed above, we apply Lemma \ref{LemCle} to get $\de >0$
  satisfying its conclusion. By the compactness of $[a,b]$ and the
  continuity of $\ga$, 
  there is $\eta>0$ such that  $t\in[a,b]$, $d(x,\ga(t)) < \eta$
  imply $|x-\ga(t)| < \de.$
 Since  $\ga_n$ converges $d$-uniformly to $\ga$,
  there exists an integer $n_0$ such that, for each $n\ge n_0$
  we have $d(\ga_n(t),\ga(t)) < \eta$ for each $t\in[a, b]$.
Hence, for each $n\ge n_0$ and almost all $t \in E_C$, we have
  \[L (\ga_n(t),\zeta_n(t))\ge L (\ga(t), \zeta(t)) +
    \partial_v L(\ga(t), \zeta(t))\cdot(\zeta_n(t) -\zeta(t))-\ep,\]
 and from the uniform superlinearity, $L(\ga_n(t), \zeta_n(t))\ge C(0)$
 a. e.where. Thus
  \begin{align}\nonumber
    A_{L,\lam_n}(\ga_n)&\ge \int_{E_C}e^{\lam_n t}L(\ga(t), \zeta(t))\ dt
                       +C(0)\int_{[a,b]\entre E_C}e^{\lam_n t}\ dt\\
& +\int_{E_C}e^{\lam_n t}\partial_vL(\ga(t), \zeta(t))\cdot (\zeta_n(t) -\zeta(t))\ dt
-\ep e^{\lam b}m(E_C) \label{1}
  \end{align}
  Since $\{\|e^{\lam_n t}\zeta_n(t)\|\}$ is uniformly integrable, $e^{\lam_n t}\zeta_n(t)$ converges to
  $e^{\lam t}\zeta(t)$ in the weak topology $\sigma(L^1,L^\infty)$.
Since $\|\zeta(t)\|\le C$ for $t\in E_C$, the function $\chi_{E_C}(t)\partial_vL(\ga(t), \zeta(t))$
is bounded. Thus
\[\int_{E_C}\partial_v L(\ga(t), \zeta(t))\cdot (e^{\lam_n t}\zeta_n(t) -e^{\lam t}\zeta(t))\ dt,
\int_{E_C}\partial_v L(\ga(t), \zeta(t))\cdot (e^{\lam t}-e^{\lam_n t})\zeta(t))\ dt\to 0,\]
as $n\to\infty$. Taking limit in \eqref{1} we have
\[l=\lim_{n\to\infty}A_{L,\lam_n}(\ga_n)\ge \int_{E_C}e^{\lam t}L(\ga(t), \zeta(t))\ dt
  +C(0) \int_{[a,b]\entre E_C}e^{\lam t}\ dt-\ep e^{\lam b}  m(E_C).\]
Letting $\ep\to 0$ we have
\begin{equation}
  \label{eq1}
  l=\lim_{n\to\infty}A_{L,\lam_n}(\ga_n)\ge \int_{E_C}e^{\lam t}L(\ga(t), \zeta(t))\ dt
  +C(0) \int_{[a,b]\entre E_C}e^{\lam t}\ dt.
\end{equation}
Since $\zeta(t)$ is defined and finite for almost all $t\in[a,b]$ we
have that $E_C\nearrow E_\infty$ as $C\nearrow +\infty$ with
$m([a,b]\entre E_\infty)=0$. Since $L(\ga(t), \zeta(t))$ is bounded below by $C(0)$,
the monotone convergence theorem gives
\[\int_{E_C}e^{\lam t}L(\ga(t), \zeta(t))\ dt\to\int_a^be^{\lam t}L(\ga(t), \zeta(t))\ dt \hbox{ as }C\to+\infty.\]
Letting $C\nearrow+\infty$ in \eqref{eq1} we finally obtain
\[l=\lim_{n\to\infty}A_{L,\lam_n}(\ga_n) \ge A_{L,\lam}(\ga)\]
\end{proof}
\begin{corollary}\label{low-semi}
  The action
  $A_{L,\lam}:\cWD([a,b])\to \R\cup\{+\infty\}$ is lower semicontinuous
  for the topology of $d$-uniform convergence in $\cWD([a,b])$.
\end{corollary}
\begin{proof}
  Let $\ga_n$ be a sequence in $\cWD([a,b])$ that converges
  $d$-uniformly to $\ga\in \cWD([a,b])$. We must show that
  \[\liminf_{n\to\infty}A_{L,\lam}(\ga_n)\ge A_{L,\lam}(\ga).\]
If $\liminf\limits_{n\to\infty}A_{L,\lam}(\ga_n)=+\infty$ there is nothing to prove.
In other case, the result follows from Theorem \ref{Cle}
\end{proof}

\begin{corollary}[Tonelli's Theorem]\label{tonelli}
  If $K\subset M$ is compact, $c\in\R$
  then the set 
\[\cF(K,c)=\{\ga\in\cWD([a,b]): \ga([a,b])\cap K\ne\emptyset,
A_{L,\lam}(\ga)\le c\}\]
is a compact subset of $\cWD([a,b])$ for the topology of
$d$-uniform convergence.
\end{corollary}
\begin{proof}
By the compactness of $K$ and Theorem \ref{Cle}, the subset
  $\cF(K,c)$ is closed in the space of continuous curves $C([a,b],M)$.
  By the uniform superlinearity,
  $L(v)\ge\|v\|+C(1)$ for all $v\in\cD$. Thus
  for any $\ga\in\cWD([a,b])$ and any $t,s\in[a,b]$ with $t\le s$, we have 
 \[C(1)\int_t^se^{\lam r}dr+e^{\lam a}\int_t^s\|\dga\|\le A_{L,\lam}(\ga)\]
If $\ga\in\cF(c,K)$ we have
\[d(\ga(s),\ga(t)\le e^{-\lam a}\Big(c+|C(1)|\int_a^be^{\lam r}dr\Big),\]
and thus, letting $r= e^{-\lam a}(c+|C(1)|\int_a^be^{\lam r}dr)$, we get
\[\ga([a,b])\subset\{y\in M:d(y,K)\le r\}.\]
Since $\cF(c)$ is $d$-absolutely equicontinuous by  Theorem \ref{UI} and
Remark \ref{basic} (1),
the Arzelá-Ascoli Theorem implies that $\cF(K,c)$ is a compact subset
of $\cWD([a,b])$ for the topology of $d$-uniform convergence.
\end{proof}
Denote $\cC_{a,b}(x,y):=\{\ga\in\cWD([a,b]): \ga(a)=x, \ga(b)=y\}$.
\begin{corollary}[Tonelli minimizers]\label{ton-min}
    For each $x,y\in M$ 
 and each $a,b\in\R$, $a<b$, there exists $\ga\in\cC_{a,b}(x,y)$,
 $A_{L,\lam}(\ga)\le A_{L,\lam}(\af)$ for any $\af\in\cC_{a,b}(x,y)$.
\end{corollary}
\begin{proof}
  Set $\bar C=\inf\{A_{L,\lam}(\ga):\ga\in\cC_{a,b}(x,y)\}$.
  By Corollary \ref{tonelli}, the set 
\[\{\ga\in\cC_{a,b}(x,y):A_{L,\lam}(\ga)\le\bar C+1\}\]
is a compact subset of $\cWD([a,b])$ for the topology of
$d$-uniform convergence.
Choose $\ga_n\in\cC_{a,b}(x,y)$ such that $A_{L,\lam}(\ga_n)<\bar C+\dfrac 1n$.
Then $\ga_n$ has a subsequence that converges $d$-uniformly
to some $\ga\in \cC_{a,b}(x,y)$. By Theorem \ref{Cle}, $A_{L,\lam}(\ga)=\bar C$. 
\end{proof}

\section{The discounted value function}
\label{sec:disc-value}
For $\lambda>0$ we define the discounted value function
\begin{equation}\label{eq:value}
  u_\lam(x)=\inf_{\ga(0)=x}\int_{-\infty}^0 e^{\lam t}L(\dga(t))\ dt
\end{equation}
where the infimum is taken over the curves $\ga\in\cWD(]-\infty,0])$,
with $\ga(0)=x$.
\begin{proposition} \label{value-prop}
  The discounted value function has the following properties
  \begin{enumerate}
  \item 
   $\min L\le\lam u_\lam\le A(0)$.
  \item 
  For any $\af\in\cWD([a,b])$
  \begin{equation}
    \label{eq:lam-dom}
      u_\lam(\af(b))e^{\lam b}-u_\lam(\af(a))e^{\lam a}\le A_{L,\lam}(\af)
\end{equation}
    \end{enumerate}
 \end{proposition}
\begin{proof}1)  Taking $\ga(t)\equiv x$ we have
  \[u_\lam(x)\le \int_{-\infty}^0 e^{\lam t}A(0)\ dt=\frac{A(0)}{\lam}\]
Given $\ep>0$ take $\ga\in\cWD(]-\infty,0])$ with $\ga(0)=x$ such that
\[\frac{\min L}{\lam}\le\int_{-\infty}^0e^{\lam t}L(\dga(t))\ dt<
  u_\lam(x)+\ep.\]
2) For $\af\in\cWD([a,b])$ define $\af_b(t)=\af(t+b)$, then \eqref{eq:lam-dom}
is equivalent to
\[ u_\lam(\af_b(0))-u_\lam(\af_b(a-b))e^{\lam (a-b)}\le A_{L,\lam}(\af_b),\]
and so we can assume that $b=0$. Given $\ga\in\cWD([-\infty,0])$ with $\ga(0)=\af(a)$
define $\ga_a(s)=\ga(s-a)$. From the definition of $u_\lam$ we have
\[  u_\lam(\af(0))\le \int_{-\infty}^a e^{\lam s}L(\dga_a(s))\ ds+A_{L,\lam}(\af)
  =e^{\lam a}\int_{-\infty}^0 e^{\lam t}L(\dga(t))\ dt+A_{L,\lam}(\af)\]
Taking the infimum over the curves $\ga\in\cWD(]-\infty,0])$,
with $\ga(0)=\af(a)$ we get 
 \[u_\lam(\af(0))\le u_\lam(\af(a))e^{\lam a}+A_{L,\lam}(\af)\]
\end{proof}

\begin{theorem} [Dynamic Programming Principle]\label{dpp}
The function $u_\lam$ satisfies
\begin{equation}\label{eq:dpp}
  u_\lam(x)=\inf\{u_\lam(\ga(-T))e^{-\lam T}+A_{L,\lam}(\ga):
\ga\in\cWD([-T,0]),\ga(0)=x\}.
\end{equation}
Moreover the infimum is attained
\end{theorem}
\begin{proof}
  The inequality $\le$ in \eqref{eq:dpp} follows from \eqref{eq:dpp}.
 Given $\ep>0$ take $\ga\in\cWD(]-\infty,0])$ with $\ga(0)=x$ such that
\[\int_{-\infty}^0e^{\lam t}L(\dga(t))\ dt< u_\lam(x)+\ep.\]
%Define $\ga_T(s)=\ga(s-T)$, then 
\begin{align*}
  u_\lam(\ga(-T)) e^{-\lam T}&\le% e^{-\lam T}\int_{-\infty}^0e^{\lam s}L(\dga_T(s))\ ds=
  \int_{-\infty}^{-T}e^{\lam t}L(\dga(t))\ dt=
  \int_{-\infty}^0e^{\lam t}L(\dga(t))\ dt-A_{L,\lam}(\ga|_{[-T,0]})\\&<u_\lam(x)+\ep-A_{L,\lam}(\ga|_{[-T,0]}),
\end{align*}
giving the inequality $\ge$ in \eqref{eq:dpp}.
To prove that the infimum is attained
 we consider a minimizing sequence $\ga_n\in\cWD([-T,0])$ with $\ga_n(0)=x$.
For $n$ large enough
\[A_{L,\lam}(\ga_n)\le 1+u_\lam(x)-u_\lam(\ga_n(-T))e^{-\lam T}\le 1+2\|u_\lam\|_\infty.\]
By Corollary \ref{tonelli}, up to subsequences, $\ga_n$ converges
$d$-uniformly to $\ga\in\cWD([-T,0])$ with $\ga(0)=x$ and by Theorem \ref{Cle}
\[A_{L,\lam}(\ga)\le\liminf_{n\to\infty}A_{L,\lam}(\ga_n)=u_\lam(x)-u_\lam(\ga(-T))e^{-\lam T}\]
\end{proof}
\begin{proposition}\label{lip-value}
  $u_\lam$ is Lipschitz for the metric $d$, uniformly in $\lam >0$
\end{proposition}
\begin{proof}
Let $x,y\in M$, $d=d(x,y)$  and take a minimizing
geodesic $\ga\in\cWD(]-d,0])$ with $\ga(-d)=y$, $\ga(0)=x$ and $\|\dga\|=1$ a.e.

\[u_\lam(x)\le u_\lam(y)e^{-\lam d}+
\int_{-d}^0e^{\lam s}L(\dga(s))ds\le  u_\lam(y)e^{-\lam d}+\frac{1-e^{-\lam d}}\lam (A(1))\]
\[u_\lam(x)-u_\lam(y)\le \frac{1-e^{-\lam d}}\lam (-\lam u_\lam (y)+A(1))
\le \frac{1-e^{-\lam d}}\lam (A(1) -\min L)\le d(A(1) -\min L)\]

\end{proof}

\begin{proposition}\label{min-inf}
Given $\lam>0$,  $x\in M$ there exists $\ga_{\lam,x}\in\cWD(]-\infty,0])$ such
that $\ga_{\lam,x}(0)=x$ and for any $t\ge 0$
\begin{equation}
  \label{eq:lam-calib}
u_\lam(x) = u_\lam(\ga_{\lam,x}(-t))e^{-\lam t} + A_{L,\lam}(\ga_{\lam,x}|_{[-t,0]}).
\end{equation}
Moreover $\|\dga_{\lam,x}\|_\infty$ is bounded uniformly in
$\lam>0,x\in M$ and
\begin{equation}
  \label{eq:min-inf}
  u_\lam(x)=\int_{-\infty}^0e^{\lam t}L(\dga_{\lam,x}(t))\ dt.
  \end{equation}
\end{proposition}
\begin{proof}
By Theorem \ref{dpp}, for each $T>0$
there is  $\ga^T\in\cWD([-T,0])$ with $\ga^T(0)=x$ and
\[u_\lam(x)=u_\lam(\ga^T(-T))e^{-\lam T}+A_{L,\lam}(\ga).\]
Since $u$ satisfies \eqref{eq:lam-dom}, for any $t\in [0,T]$ we have 
\[u_\lam(x)-e^{-\lam t}u_\lam(\ga^T(-t))=A_{L,\lam}(\ga^T|_{[-t,0]}).\]
As in the proof Teorem\ref{dpp}
  \[A_{L,\lam}(\ga^T|_{[-t,0]})\le 2\|u_\lam\|_\infty\]
    By Tonelli's Theorem there is a sequence $T_j\to\infty$
such that $\ga^{T_j}|[-t,0]$ converges uniformly with the metric
$d$. Aplying this argument to a sequence $t_k\to\infty$ and using a
diagonal trick one gets a sequence $\sigma_n\to\infty$ and a curve
$\ga_{\lam,x}\in\cWD(]-\infty,0])$ such that
$\ga^{\sigma_n}|[-t,0]$ converges $d$-uniformly 
to $\ga_{\lam,x}|[-t,0]$  for any $t>0$. By the continuity of $u$,   for any $t>0$,
$u_\lam(\ga_{\lam,x}(-t))=\lim\limits_{n\to\infty} u_\lam(\ga^{\sigma_n}(-t))$.
By Theorem \ref{Cle},
\[u_\lam(x)=u_\lam(\ga_{\lam,x}(-t))e^{-\lam t}+A_L(\ga_{\lam,x}|_{[-t,0]}).\]
By Proposition \ref{lip-value} there is a uniform Lipschitz constant $K$ for $u_\lam$.
By the superlinearity, $L(v)=(K+1)\|v\|+C(K+1)$. For $a<b<0$
  \begin{align*}
u_\lam(\ga_{\lam,x}(b))e^{\lam b}-u_\lam(\ga_{\lam,x}(a))e^{\lam a}
    &= A_{L,\lam}(\ga_{\lam,x}|_{[a,b]})\\&\ge e^{\lam a}(K+1)\int_a^b\|\dga_{\lam,x}\|
    +C(K+1)\frac{e^{\lam b}-e^{\lam a}}\lam.\\
    u_\lam(\ga_{\lam,x}(b))e^{\lam b}-u_\lam(\ga_{\lam,x}(a))e^{\lam a}
    &\le u_\lam(\ga_{\lam,x}(b))(e^{\lam b}-e^{\lam a})
      +Ke^{\lam a}d(\ga_{\lam,x}(b),\ga_{\lam,x}(a))\\
&\le A(0)\frac{e^{\lam b}-e^{\lam a}}\lam+K e^{\lam a}\int_a^b\|\dga_{\lam,x}\|
  \end{align*}
  Thus
  \[\frac 1{b-a}\int_a^b\|\dga_{\lam,x}\|\le (A(0)-C(K+1))
    \frac{e^{\lam(b-a)}-1}{\lam(b-a)}\]
  implying that $\|\dga_{\lam,x}\|\le (A(0)-C(K+1)$ a.e.
Letting $t\to+\infty$ in \eqref{eq:lam-calib} we get \eqref{eq:min-inf} by the dominated
convergence theorem. 
\end{proof}\section{Weak KAM theory}
\label{sec:wkam}
In this section we extend weak KAM theory to our setting. The
difference with the standard case is the lack of a Lagrangian dynamics and Tonelli's
theorem takes care of this difficulty. Denote  $A_{L,0}$ by  $A_{L}$
\subsection{The weak KAM theorem}
\label{sec:wkam-thm}
\begin{definition} Let $c\in\R$, we say that 
  \begin{itemize}
  \item $u:M\to\R$ is $L+c$ {\em dominated} if for any $\ga\in\cWD([a,b])$
we have 
 \[ u(\ga(b))-u(\ga(a)) \le A_{L+c}(\ga)\]
\item $\ga\in\cWD([a,b])$ is {\em calibrated} by an $L+c$ dominated
  function $u:M\to\R$  if
  \[ u(\ga(b))-u(\ga(a)) = A_{L+c}(\ga)\]
  \item $u:M\to\R$ is a {\em backward (forward)  c-weak KAM solution} if it
  is $L+c$ dominated and for any $x\in M$  there is
  $\ga\in\cWD(]-\infty,0])$ ($\ga\in\cWD([0,\infty))$) 
    calibrated by $u$ such that $\ga(0)=x$.
\end{itemize}
\end{definition}
Suppose $M$ is compact. Consider the value function $u_\lam$, $\lam>0$.
Since $\lam u_\lam$ is uniformly bounded,  $u_\lam$ is uniformly
Lipschitz there is a sequence $\lam_n$ converging to zero such
that $\lam_nu_{\lam_n}$ converges to a constant $-c$. Thus $u_\lam- \max u_\lam$ is
uniformly bounded and Lipschitz, and then, trough some sequence converges to a  
Lipschitz function $v$.
\begin{theorem}[weak KAM theorem]\label{wkam-teo}
Suppose $M$ is compact.
  Let $v_\lam=u_\lam - \min u_\lam$.
Suppose that $v=\lim\limits_{k\to\infty}v_{\lam_k}$ and
$c=-\lim\limits_{k\to\infty}\lam_k\min u_{\lam_k}$ for a sequence $\lam_k\to 0$.
Then $v$ is a backward $c$-weak KAM solution.
\end{theorem}
\begin{proof}
 Let $\ga\in\cWD([a,b])$, By \eqref{eq:dpp}, for any $\lam>0$ 
\[v_\lam(\ga(b)) e^{\lam a}-v_\lam(\ga(b))e^{\lam a}\le (e^{\lam a}-e^{\lam b})\min u_\lam+A_{L,\lam}(\ga).\]
  Taking limits $\lam_k\to 0$,
  \[v(\ga(b)) -v(\ga(a)) \le(b-a)c+A_L(\ga),\]
giving that $v$ is $L+c$ dominated.

Let $x\in M$ and $\ga_\lam\in\cWD(]-\infty,0])$ be a curve such that $\ga_\lam(0)=x$
for any $t\ge 0$
\[u_\lam(x) = u_\lam(\ga(-t))e^{-\lam t} +A_{L,\lam}(\ga_\lam),\]
so that
\[v_\lam(x) = v_\lam(\ga(-t))e^{-\lam t} +(e^{-\lam t}-1)\min u_\lam+A_{L,\lam}(\ga_\lam).\]
Then, as in the proof of Proposition \ref{min-inf}, there is a
subsequence $(\lam^n)_n$ of $(\lam_k)_k$ and a curve $\ga\in\cWD(]-\infty,0])$ such for any $t>0$,
$\ga_{\lam^n}|_{[-t,0]}$ converges $d$-uniformly to
$\ga|_{[-t,0]}$. Taking limits $\lam^n\to 0$, by Theorem \ref{Cle} we have 
\begin{equation}
 \label{lim-calib}
v(x) = u(\ga(-t))+ct+A_L(\ga)
\end{equation} 
showing that $\ga$ is calibrated by $v$. Thus $v$ is a c-weak KAM solution.
\end{proof}
Suppose there is a function $u:M\to\R$ that is $L+c$ dominated
and let $x\in M$. Taking $\ga:[a,b]\to M$ the constant curve
$\ga(t)=x$ we have $L(x,0)+c\ge 0$. Thus $c\ge -\min\limits_{x\in M} L(x,0)$.
Define $c(L)=\inf\{c\in\R:\exists u:M\to\R\ L+c\hbox{ dominated }\}$.

If $k\ge-\min L$, then $L+k\ge 0$ and so any constant function is $L+k$ dominated.
Thus $c(L)\le -\min L$.
\begin{example}
Let $M$ be compact,  $U:M\to\R$ be smooth and  $L(x,v)=\|v\|^2-U(x)$.
Then $\min L=\min\limits_{x\in M} L(x,0)=-\max U$. 
Thus $c(L)=\max U$. 
\end{example}
\begin{proposition}\quad
  \begin{enumerate}
     \item  There exists $u:M\to\R$ that is $L+c(L)$ dominated.
     \item If $M$ is compact and there is a bakward c-weak KAM solution
       $u:M\to\R$, then $c=c(L)$.
       \end{enumerate}
     \end{proposition}
The proof of (1) is as the one of Theorem 4.2.7 in \cite{F}.
 The proof of (2) is as the one of Corollary 4.3.7 in \cite{F}.

\begin{proposition}
  Let $u:M\to\R$ be $L+c$ dominated then
  \begin{enumerate}
  \item $|u(y)-u(x)|\le (A(1)+c) d(x,y)$
  \item If $\ga\in\cWD(I)$ then $u\circ\ga$ is absolutely continuous and
    $(u\circ\ga)'(t)\le L(\dga(t))+c$ at any $t\in I$ where
    $u\circ\ga$ and $\ga$ are  differentiable.
  \end{enumerate}
\end{proposition}
\begin{proof}
  (1) Taking a minimizing geodesic $\af\in\cC_d(x,y)$ with
  $\|\dot\af\|=1$ a.e. we have 
\[u(y)-u(x)\le A_{L+c}(\af)\le (A(1)+c)\ d(x,y).\]

(2) Since $u$  is Lipschitz with constant $K=A(1)+c$, given $\ep>0$
there is $\de>0$ such that 
  for any disjoint $ ]a_1,b_1[,\cdots,]a_k,b_k[\subset I$,
  \begin{align*}
    \sum_{i=1}^kb_i-a_i<\de\implies &\\
   \sum_{i=1}^k|u\circ\ga(b_i)-u\circ\ga(a_i)|&< K\sum_{i=1}^k d(\ga(a_i),\ga(b_i))
  \le K\sum_{i=1}^k \int_{a_i}^{b_i}\|\dga\|<\ep.
\end{align*}
Let $t\in I$ be a point of differentiability of $u\circ\ga$ and $\ga$, dividing
\[u(\ga(t+h))-u(\ga(t))\le A_{L+c}(\ga|_{[t,t+h]})\]
by $h$ and letting $h\to 0$ we get $(u\circ\ga)'(t)\le L(\dga(t))+c$.
\end{proof}
\subsection{The Lax-Oleinik semigroup}
\label{sec:lax-oleinik}
Set $\cC_t(x,y)=\cC_{0,t}(x,y)$ and $\cC(x,y)=\bigcup\limits_{t>0}\cC_t(x,y)$.

  \begin{definition} We define the
  {\em minimal action} to go from $x$ to $y$ in time $t>0$ as the function
  $\bfh:M\times M\times\R^+\to\R$ given by
  \[\bfh(x,y,t):=h_t(x,y):=\inf\{A_L(\ga):\ga\in\cC_{t}(x,y)\}\]
\end{definition}
\begin{proposition}\label{prop-ht} The minimal action has the
    following properties
    \begin{enumerate}
    \item $h_t(x,y)\ge t \inf\limits_{TM}L$
    \item $h_{t+s}(x,z)=\inf\limits_{y\in M}h_t(x,y)+h_s(y,z)$
      \item $\ga\in\cWD([a,b])$, $a<b$ is minimizing if and only if
$h_{b-a}(\ga(a),\ga(b))=A_L(\ga)$.
\item For each  $x,y\in M, t>0$, there exists $\ga\in\cC_{t}(x,y)$
  such that $h_t(x,y)=A_L(\ga)$.
    \end{enumerate}
  \end{proposition}
  \begin{proposition}
    \label{htcont}
    Assume that there are $C_1, C_2\ge 1$ such that
    \begin{equation}
      \label{vlvbound}
          \partial_v L\cdot v\le C_1 L+C_2
\end{equation}
        and let $\De_\de=\{x, y\in M:d(x,y)\le 1/\de\}$ for $\de>0$. Then $\bfh$
is Lipschitz on $\De_\de\times[\de,1/\de]$. In particular, $\bfh$ is continuos
\end{proposition}
\begin{proof}
First observe that $\bfh$ is bounded on $\De_\de\times[\de,1/\de]$.
Indeed, let $d=d(x,y)$  and take a minimizing
geodesic $\sigma\in\cC_d(x,y)$ with $\|\dot\sigma\|=1$ a.e. and
$\ga(s)=\sigma(sd(x,y)/t)$, then $\|\dot\ga\|=d(x,y)/t\le \de^{-2}$
for  $(x,y)\in\De_\de$, $t\ge\de$. Thus 
$\bfh(x,y,t)\le A(\de^{-2})/\de$.

  Next observe that the assumption \eqref{vlvbound} implies that
  \[L(sv)\le s^{C_1}L(v)+\frac{C_1}{C_2}(s^{C_1}-1),\quad s\ge 1.\]
Let $x,y\in M$, $t>0$. For $z\in M$, $d=d(y,z)$, take a minimizing
geodesic $\sigma\in\cC_d(y,z)$ with $\|\dot\sigma\|=1$ a.e.
For $\ga\in\cC_t(x,y)$ define $\bga\in\cC_{t+d}(x,z)$ by
$\bga|[0,t]=\ga$, $\bga(s)=\sigma(s-t)$ for $s\in[t,t+d]$. Thus
$\bfh(x,z,t+d)\le A_L(\bga)\le A_L(\ga)+A(1)d$ and then
\begin{equation}
  \label{eq:1}
\bfh(x,z,t+d(y,z))\le \bfh(x,y,t) +A(1)d(y,z).  
\end{equation}
Similarly, defining $\hat\ga\in\cC_\tau(y,z)$, $\tau>t$ by
$\hat\ga|[0,t]=\ga$, $\hat\ga(s)=y$ for $s\in[t,\tau]$ we get
\begin{equation}
  \label{eq:2}
\bfh(x,y,\tau)\le \bfh(x,y,t)+A(0)(\tau-t).  
\end{equation}
Let $r>0$ an $\ga\in\cC_{t+r}(x,y)$ and define $\bga\in\cC_t(x,y)$
by $\bga(s)=\ga((t+r)s/t)$ we have
\begin{align*}
  \bfh(x,y,t)&\le A_L(\bga)=\frac t{t+r}\int_0^{t+r}L\Bigl(\frac {t+r}t\dga\Bigr)\\
&\le \Bigl(\frac {t+r}t\Bigr)^{C_1-1}A_L(\ga)+\frac{C_1}{C_2}\left(\Bigl(\frac {t+r}t\Bigr)^{C_1}-1\right),
\end{align*}
and then
\begin{equation}
  \label{eq:3}
   \bfh(x,y,t)\le \Bigl(\frac {t+r}t\Bigr)^{C_1-1} \bfh(x,y,t+r)+\frac{C_1}{C_2}\left(\Bigl(\frac {t+r}t\Bigr)^{C_1}-1\right).
 \end{equation}
 Thus
 \begin{align*}
 \bfh(x,y,t)-\bfh(x,y,t+r)&\le\Bigl(\Bigl(\frac {t+r}t\Bigr)^{C_1-1} -1\Bigr)\bfh(x,y,t+r)+\frac{C_1}{C_2}\left(\Bigl(\frac {t+r}t\Bigr)^{C_1}-1\right)\\
&\le C_1 r(1+\de^{-2}) ^{C_1-1}\Bigl(\frac{A(\de^{-2})}\de+\frac{C_1}{C_2}\Bigr).
 \end{align*}
 Letting $d=d(y,z)$, from \eqref{eq:1} , \eqref{eq:3} we get
 \begin{align*}
\bfh(x,z,t)-\bfh(x,y,t)&\le\Bigl(\Bigl(\frac {t+d}t\Bigr)^{C_1-1} -1\Bigr)\bfh(x,z,t+d)+\frac{C_1}{C_2}\left(\Bigl(\frac {t+d}t\Bigr)^{C_1}-1\right)\\
&+\bfh(x,z,t+d)-\bfh(x,y,t)\\&\le \Bigl(C_1 (1+\de^{-2})
   ^{C_1-1}\Bigl(\frac{2A(\de^{-2})}\de+\frac{C_1}{C_2}\Bigr)+A(1)\Bigr)\ d(y,z),
 \end{align*}
 and similarly for $\bfh(z,y,t)-\bfh(x,y,t)$.
\end{proof}
\begin{corollary}\label{htlip}
  Assume $M$ is compact and there are $C_1, C_2>0$ such that
  \eqref{vlvbound} holds.
  For any $\de>0$ there is $K_\de>0$ such that
  $h_t$ is $K_\de$-Lipschitz for any $t\ge\de$.
\end{corollary}
\begin{proof}
  Since $\bfh$ is continuous and $M$ is compact, for $x,y\in M$, $t,s>0$
  there is $z\in M$ such that $h_{t+s}(x,y)=h_t(x,z)+h_s(z,y)$.

 Let $\de>0$, by Proposition \ref{htcont} there is $K_\de>0$ such that
 $\bfh$ is $K_\de$ Lipschitz on $M\times M\times[\de,1/\de]$. If $t>1/\de$ choose $t_0<t_1<\cdots<t_k=t$ such that
$t_i-t_{i-1}\in[\de,1/\de]$. Let $x_1, x_2, y_1, y_2\in M$ and suppose 
$h_t(x_1, y_1)\ge h_t(x_2, y_2)$. There are $z_0=x_2, z_1,\ldots,
z_t=y_2\in M$ such that 
$h_t(x_2, y_2)=\displaystyle\sum_{i=1}^lh_{t_i-t_{i-1}}(z_{i-1},z_i)$.
Since
\[h_t(x_1, y_1)\le h_{t_1}(x_1,z_1)+\sum_{i=2}^{l-1}h_{t_i-t_{i-1}}(z_{i-1},z_i)
+h_{t-t_{l-1}}(z_{l-1},y_1),\]
we have
\begin{align*}
  h_t(x_1, y_1)- h_t(x_2, y_2)&\le h_{t_1}(x_1, z_1)- h_{t_1}(x_2, z_1)
  +h_{t-t_{l-1}}(z_{l-1}, y_1)-h_{t-t_{l-1}}(z_{l-1}, y_2)\\
                              &\le K_\de (d(x_1, x_2)+d(y_1, y_2))
\end{align*}
\end{proof}

\begin{definition}[Lax-Oleinik semi-group]
For $u:M\to[-\infty,+\infty]$ and $t>0$ we define $\cL_tu:M\to[-\infty,+\infty]$ by
\[\cL_tu(x)=\inf\{u(\ga(0))+A_L(\ga):\ga\in\cWD([0,t]),\ \ga(t)=x\}\]
\end{definition}
Notice that
\[\cL_tu(x)=\inf_{y\in M}u(y)+h_t(y,x)\]
  \begin{proposition}\label{prop-lax}
The family of maps $\{\cL_t\}_{t\ge 0}$ has the
 following properties
 \begin{enumerate}
\item $u:M\to\R$ is $L+c$ dominated if and only if $u\le \cL_tu+ct$ for all $t\ge 0$.
  \item For any $s,t>0$, $\cL_{t+s}=\cL_t\circ \cL_s$
\item If $u:M\to\R$ is Lipschitz for the metric $d$, then for each $x\in M$ and $t>0$ there is $\ga\in\cWD([0,t])$
  such that $\ga(t)=x$ and
  \[\cL_tu(x)=u(\ga(0))+A_L(\ga)=u(\ga(0))+h_t(\ga(0),x).\]
\item If $M$ is compact,  $u\in C(M)$ and that there are
$C_1, C_2>0$ such that \eqref{vlvbound} holds, then for each $x\in M$,
$t>0$ there is $y\in M$ such that $\cL_tu(x)=u(y)+h_t(y,x)$.
\end{enumerate}
\end{proposition}
\begin{proof}
  (3) Using the constant curve with value $x$ we have 
$\cL_tu(x)\le u(x)+t A(0).$
Thus 
\[\cL_tu(x)=\inf\{u(\ga(0))+A_L(\ga):\ga\in\cF(u,x,t)\}\]
where 
\[\cF(u,x,t)=\{\ga\in\cWD([0,t]), \ga(t)=x,
   u(\ga(0))+ A_L(\ga)\le u(x)+t A(0)\}.\]
If $u$ is $K$-Lipschitz, $u(x)-u(\ga(0))\le K\ d(x,\ga(0))$.

By the superlinearity of $L$ we have 
\[C(K+1) t+(K+1) \ell(\ga)\le A_L(\ga),\]
thus, if $\ga\in\cF(u,x,t)$ then $ d(x,\ga(0))\le \ell(\ga)\le
t(A(0)-C(K+1))$. Therefore
\[\cL_tu(x)=\inf\{u(\ga(0))+A_L(\ga): \ga\in\cF'(x,t,K) \}\]
where 
\[\cF'(x,t,K)=\{\ga\in\cWD([0,t]): \ga(t)=x , A_L(\ga)\le t (A(0)+K (A(0)-C(K+1)))\} \]
By Tonelli's Theorem the set $\cF'(x,t,K)$
is compact for the topology of uniform convergence with the metric $d$.
Choose  $\ga_n\in\cF'(x,t,K)$ such that
$u(\ga_n(0))+A_L(\ga_n)<\cL_tu(x)+\dfrac 1n$. 
Then $\ga_n$ has a subsequence $\ga_{n_j}$
that converges uniformly with the metric $d$ to some $\ga\in\cF'(x,t,K)$.
By continuity of $u$, $u(\ga(0))=\lim u(\ga_{n_j}(0))$.
By Theorem \ref{Cle}, $u(x)=\cL_tu(x)=u(\ga(0))+A_L(\ga)$.
\end{proof}
\begin{proposition}\label{weak=fix}
Let $c\in\R$. A function $u:M\to\R$ satisfies $u=\cL_tu+c t$ for all $t\ge 0$
if and only it is a weak KAM solution.
\end{proposition}
\begin{proof}
  It is clear that a weak KAM solution $u$ satisfies $u=\cL_tu+c t$ for all $t\ge 0$.
Assuming that $u=\cL_tu+c t$ for all $t\ge 0$, by item (1) in Proposition
\ref{prop-lax},  we have that $u$ is $L+c$ dominated.

Let $x\in M$ by item (3) of Proposition \ref{prop-lax} for each $s>0$
there is  $\ga_s\in\cWD([-s,0])$ with $\ga(0)=x$ and
$\cL_su(x)=u(\ga_s(-s))+A_L(\ga)$; then  $u(x)=u(\ga_s(-s))+A_{L+c}(\ga)$.
Since $u$ is $L+c$ it is dominated, it is $K$-Lipischitz for $K=c+A(1)$ we have that
$u(x)=u(\ga_s(-t))+A_{L+c}(\ga_s|[-t,0])$ for any $t\in[0,s]$.
From the proof of item (3) of Proposition \ref{prop-lax}, for $s>t$
  we have
  \[\int_{-t}^0L(\dga_s)\le t (A(0)+K (A(0)+C(K+1))).\]
    By Tonelli's Theorem there is a sequence $s_j\to\infty$
such that $\ga_{s_j}|[-t,0]$ converges uniformly with the metric
$d$. Aplying this argument to a sequence $t_k\to\infty$ and using a
diagonal trick one gets a sequence $\sigma_n\to\infty$ and a curve
$\ga\in\cWD(]-\infty,0])$ such that
$\ga_{\sigma_n}|[-t,0]$ converges uniformly with the metric
$d$ to $\ga|[-t,0]$  for any $t>0$. By the continuity of $u$,   for any $t>0$,
$u(\ga(-t))=\lim\limits_{n\to\infty} u(\ga_{\sigma_n}(-t))$.
By Theorem \ref{Cle}, $\cL_tu(x)=u(\ga(-t))+A_L(\ga|[-t,0])$.
\end{proof}
\begin{corollary}\label{wkam-cor}
  Let $M$ be compact.
  For each $\de>0$ there is $C_\de>0$ such that $|h_t(x,y)+c(L)t|\le C_\de$ for $x,y\in M$,  $t\ge\de$
  \end{corollary}
\begin{proof}
 For $\ga\in\cC_{t+\de}(x,y)$ we have
\[
  A_L(\ga)=A_L(\ga|[0,1])+ A_L(\ga|[\de/2,t+\de/2])+A_L(\ga|[t+\de/2,t+\de])\ge 2\inf h_{\frac\de 2}+\inf h_t,
  \]
  thus
  \[h_{t+\de}(x,y) \ge 2 \inf h_{\frac\de 2}+\inf h_t.\]
For any $z, w\in M$
\[h_{t+\de}(x,y) \le h_{\frac\de 2}(x,z)+h_t(z,w)+h_{\frac\de 2}(w,y)\le 2\sup h_{\frac\de 2}+h_t(z,w),\]
thus

\[h_{t+\de}(x,y) \le 2 \sup h_{\frac\de 2}+\inf h_t.\]
Therefore, for $t\ge 2$ we have that
\[ h_t-\inf h_t\le 2 (\sup h_{\frac\de 2}-\inf h_{\frac\de 2}):=a\]
Let $u:M\to\R$ be such that $u=\cL_tu+c(L)t$ for all $t\ge 0$. 
%\[-\|u\|_\infty+\inf h_t\le \cL_tu=u-c(L)t\le\|u\|_\infty+\sup h_t,\]
For $t\ge\de$ we have
\begin{align*}
-2\|u\|_\infty&\le u(y)-u(x)\le h_t(x,y)+c(L)t\le a+\inf h_t+c(L)t\\
&  = a+\inf h_t +u-\cL_tu\le a+2\|u\|_\infty
  \end{align*}
  and then
  \[| h_t(x,y)+c(L)t|\le 2 (\sup h_{\frac\de 2}-\inf h_{\frac\de 2}+\|u\|_\infty).\]
\end{proof}

\subsection{Static curves and the Aubry set}
\label{sec:aubry-set}
We assume that $M$ is compact and set $c=c(L)$.
\begin{definition}
  We define
  \begin{itemize}
  \item The {\em Peierls barrier} $h:M\times M\to\R$ by
  \[h(x,y)=\liminf_{t\to\infty}h_t(x,y)+ct,\]
  The  {\em Ma\~n\'e potential} $\Phi:M\times M\to\R$
  \[\Phi(x,y)=\inf_{t>0} h_t(x,y)+ct.\]
  \end{itemize}
 \end{definition}
 It is clear that $\Phi\le h$
\begin{proposition}\label{prop-h}
   Functions $h$ and $\Phi$ have the following properties
   \begin{enumerate}
   \item $u$ is $L+c$ dominated $\iff u(y)-u(x)\le \Phi(x,y)$.
   \item For each $x\in M$, $\Phi(x,x)= 0$.
   \item $\Phi(x,z)\le \Phi(x,y)+\Phi(y,z)$.
   \item  $h(x,z)\le h(x,y)+\Phi(y,z)$,\; $h(x,z)\le \Phi(x,y)+h(y,z)$.
     \item $h$ and $\Phi$ are Lipschitz.
   \item For each $x, y\in M$ there is a sequence
     $\ga_n\in\cC_{t_n}(x, y)$ with $t_n\to\infty$ such that
     \[h(x,y)=\lim_{n\to\infty}A_L(\ga_n)+ct_n.\]
   \item   If $\ga_n\in\cWD([0,t_n])$ with $t_n\to\infty$,
     $\ga_n(0)\to x$,   $\ga_n(t_n)\to y$, then
     \[h(x,y)\le\liminf_{n\to\infty}A_L(\ga_n) +ct_n.\]
   \end{enumerate}
 \end{proposition}
 Notice that by item (4) in Proposition \ref{prop-h},
$h(x,z)=\Phi(x,z)$ if $x\in\cA$ or $z\in\cA$.

\begin{definition}  A curve $\ga\in\cC(x,y)$ is called
\begin{itemize}
\item {\em semi-static}\ if $A_{L+c}(\ga)=\Phi(x,y)$,
\item {\em static}\ if $A_{L+c}(\ga)=-\Phi(x,y)$.
\end{itemize}
The \emph{Aubry set} is the set $\cA=\{x\in M: h(x,x)=0\}$.
\end{definition}
\begin{corollary}\label{aubry}\quad
  \begin{enumerate}
\item Static curves are semi-static. 
\item A curve $\ga\in\cWD(J)$ is semi-static if and only if
\[\Phi(\ga(t),\ga(s))=\int_t^s L(\dga)+c(s-t) \hbox{ for any }t,s\in J,\ t\le s.\]
\item  A curve $\ga\in\cWD(J)$ is static if and only if
\[-\Phi(\ga(t),\ga(s))=\int_t^s L(\dga)+c(s-t) \hbox{ for any }t,s\in J,\ t\le s.\]
\item A curve calibrated by  some $L+c$ dominated function is  semi-static .
  \item Any static curve is calibrated by any $L+c$ dominated function.
\item $h(x,z)=\Phi(x,z)$ if $x\in\cA$ or $z\in\cA$.
\end{enumerate}
\end{corollary}
  \begin{example}
        Let $M$, $U$ and $L$ be as in example 1.
If $\ga:[0,t]\to U^{-1}(]-\infty,\max U-\ep])$ then
$A_{L+c(L)}(\ga)\ge\ep t$, therefore if $\ga_n\in\cC_{t_n}(x,y)$ is a
sequence such that $t_n\to\infty$ and $A_{L+c(L)}(\ga_n)$ converges to
$h(x,y)$ then for any $\ep>0$ there is $N$ such that
$\ga_n([0,t_n])\cap U^{-1}([\max U-\ep, +\infty[)\ne\emptyset$ for $n\ge N$.

If $U(x)=\max U$ and $\ga:[0,t]\to M$ is the constant curve $x$, then 
$A_{L+c(L)}(\ga)=0$ and thus $h(x,x)=0$.

If $U(x)<\max U$ there are $\de,\ep>0$ such that $U(y)\le\max U-\ep$
for $d(x,y)\le\de$. It is not hard to see that for any $\ga\in\cC_t(x,x)$
such that $\ga([0,t])\cap U^{-1}([\max U-\ep, \max U])\ne\emptyset$ we have 
$A_{L+c(L)}(\ga)\ge 2\sqrt{\de\ep}$ and then $h(x,x)\ge 2\sqrt{\de\ep}$.

Thus $\cA=U^{-1}(\max U)$.
\end{example}
The following Proposition says
that semi-static curves have energy $c(L)$.
\begin{proposition}\label{energy}
  Let $\eta\in\cWD(J)$ be semi-static. For a. e. $t\in J$
\[E(\dot\eta(t))=L_v(\dot\eta(t))\dot\eta(t) - L(\dot\eta(t))=c\]
 \end{proposition}
  \begin{proof}
 For $\lam>0$, let $\eta_\lam(t):=\eta(\lam t)$ so that
$\dot\eta_\lam(t)=\lam\dot\eta(\lam t)$ a. e.

For $r,s\in J$ let
$$A_{rs}(\lam):=\int_{r/\lam}^{s/\lam}[L(\dot\eta_\lam(t))+c]\, dt
=\int_r^s[L(\lam\dot\eta(t))+c]\,\frac{dt}{\lam}. $$
Since $\eta$ is a free-time minimizer, differentiating
$ A_{rs}(\lam)$ at $\lam=1$, we have that
  $$  0=A_{rs}'(1)=
    \int_r^s[L_v(\dot\eta)\dot\eta-L(\dot\eta)-c].
  $$
Since this holds for any $r,s\in J$ we have
\[E(\dot\eta(t)) =L_v(\dot\eta(t))\dot\eta(t) - L(\dot\eta(t))=c\]
for a. e. $t\in J$.
  \end{proof}
  \begin{example}
    Let $M$, $U:M\to\R$ and $L$ be as in example 1.
    If $\ga\in\cWD([a,b])$ is semi-static, then  $\frac 12\|\dga(t)\|^2+U(\ga(t))=\max U$.
Thus 
\[A_{L+c(L)}(\ga)=\int_a^b\|\dga\|^2=\int_a^b\sqrt{2(\max U-U(\ga))}\|\dga\|
=\int_0^l\sqrt{2(\max U-U(\ga))}\ ds\]
where $s(t)=\int_a^t\|\dga\|$ is the arc length function, $s(b)=l$.

We have that $\Phi(x,y)$ is the distance from $x$ to $y$ for the Maupertuis type
sub- riemannian metric $g_x(v,w)= 2(\max U-U(x))\lip v,w\rip$.
\end{example}
  The following Propositions prove 
the existence of static curves and that $\cA\ne\emptyset$.
\begin{proposition}\label{alfa}
  Let  $u$ be $L+c$ dominated, $\ga\in\cWD(]-\infty,0])$ be calibrated by $u$.
  Take a sequence $t_n\to\infty$ such that the sequence $\ga(-t_n)$ converges.  Then
  there is a subsequence $s_k=t_{n_k}$ and $\eta\in\cWD(\R)$ such that
  \begin{enumerate}[(i)]
  \item   For  each $l>0$, $\ga(-s_k+\cdot)|[-l,l]$ converges uniformly to $\eta|[-l,l]$.
  \item $\eta$ is static.
  \end{enumerate}
\end{proposition}
\begin{proof}
  (i) Set $c=c(L)$ and let $r<l$, then
  \[\int_{-t_n+r}^{-t_n+l}L(\dga)+c(l-r)=u(\ga(-t_n+l))-u(\ga(-t_n+r))\le
    2\|u\|.\]
As we had before, Tonelli's Theorem implies that there is a sequence $n_k\to\infty$
and a curve $\eta\in\cWD(\R)$ satisfying (i).

(ii) For $l>0$
\begin{align*}
   0\le A_{L+c}\bigl(\eta|_{[-l,l]}\bigr)
    &+\Phi(\eta(l),\eta(-l))
    \\
    &\le \lim_k\bigl\{\,A_{L+c}(\ga|_{[-s_k-l,-s_k+l]})
     +\lim_m A_{L+c}(\ga|_{[-s_k+l,-s_m-l]})\,\bigr\}
    \\
    &=\lim_k\;\lim_m\;A_{L+c}(\ga|_{[-s_k-l,-s_m-l]})
    \\
    &=\lim_k\;\lim_m\;\Phi(\ga(-s_k-l),\ga(-s_m-l))
    \\
    &=\Phi(\eta(-l),\eta(-l))=0.
\end{align*}
proving that $\eta$ is static.
\end{proof}

\begin{proposition}\label{aubry=static}
   If $\ga\in\cWD([0,\infty))\cup\cWD(]-\infty,0])$ is static, then
   $p=\ga(0)\in\cA$.
\end{proposition}
\begin{proof}
  For  $\ga\in\cWD([0,\infty))$
 take a sequence $r_n\to\infty$ such that the sequence
$\ga(r_n)$ converges to $y\in M$. Then
  \begin{align*}
  0\le h(p,p)&\le h(p,y)+\Phi(y,p)   \\
        &\le \lim_n A_{L+c}(\ga|_{[0,r_n]})+\Phi(y,p)  \\
        &\le \lim_n -\Phi( \ga(r_n),p) +\Phi(y,p) =0.
  \end{align*}
  Analogously for  $\ga\in\cWD(]-\infty,0])$.
\end{proof}

\begin{proposition}\label{h-wkam}
For any $x\in M$, $h(x,\cdot)$ is a backward weak KAM solution and
$-h(\cdot,x)$ is a forward weak KAM solution.
\end{proposition}
\begin{proof} By items (1), (4) of Proposition \ref{prop-h}, $h(x,\cdot)$ is $L+c$ dominated.
 Let $\ga_n:[-t_n,0]\to M$ be a sequence of minimizing curves
connecting $x$ to $y$ such that
\[h(x,y)=\lim_{n\to\infty}A_{L+c}(\ga_n).\]
Let $l>0$, for $n$ big enough $t_n>l$  and $a(n,l)=A_{L+c}(\ga_n
[-l,0])$ is bounded from above, because if there were a sequence
$n_k\to\infty$ such that $a(n_k,l)\to+\infty$, passing to a
subsequence $\ga_{n_k}(-l)$ would converge to $z\in M$, and then we would have that
\[h(x,z)\le \liminf_{k\to\infty}A_{L+c}(\ga_{n_k}|[-t_{n_k},-l]).\]
As we had before, Tonelli's Theorem imlplies that theres is a sequence
$n_k\to\infty$ such  $\ga_{n_k}|[-l,0]$ converges uniformly with the metric
$d$ to $\ga|[-l,0]$  for any $t>0$. For $s<0$ define
$\ga(s)=\lim\limits_{k\to\infty}\ga_{n_k}(s)$.
Fix $t<0$, for $k$ large $t+t_{n_k}\ge 0$ and
\begin{equation}\label{diag}
  A_{L+c}(\ga_{n_k})=
\int\limits_{-t_{n_k}}^tL(\dga_{n_k})+c(t+t_{n_k})+
\int_t^0L(\dga_{n_k})-ct.
\end{equation}
Since  $\ga_{n_k}$ converges to $\ga$ uniformly on  $[t,0]$, we have
\[\liminf_{k\to\infty}\int_t^0L(\ga_{n_k},\dot\ga_{n_k})\ge \int_t^0L(\ga,\dot\ga).\]
From item (6) of Proposition \ref{prop-h} we have
\[h(x,\ga(t))\le\liminf_{k\to\infty}
\int_{-t_{n_k}}^tL(\ga_{n_k},\dot\ga_{n_k})+c(t+t_{n_k}).\] 

Taking $\liminf\limits_{k\to\infty}$ in \eqref{diag} we get
\[h(x,y)\ge h(x,\ga(t))+\int_t^0L(\ga,\dot\ga)-ct.\]
which implies that $\ga$ is calibrated by $h(x,\cdot)$. 
\end{proof}
\begin{corollary}\label{aubry-calibra}
  If $x\in\cA$ there exists a curve $\ga\in\cWD(\R)$ such that $\ga(0)=x$ and for
all $t\ge 0$
  \begin{align*}
    \Phi(\ga(t),x) &=h(\ga(t),x)=-\int_0^tL(\dga)-ct\\
    \Phi(x,\ga(-t)&=h(x,\ga(-t))=-\int_{-t}^0L(\dga)-ct.
  \end{align*}
  In particular the curve $\ga$ is static.
\end{corollary}

Denote by $\cK$ the family of static curves $\eta:\R \to M$, 
and for $y\in\cA$ denote  by $\cK(y)$ the set of curves $\eta\in\cK$ 
with $\eta(0)=y$. On $\cK$ we have a dynamics given by time traslation along a
static curve.
\begin{proposition}
$\cK$ is a compact metric space with respect to the uniform convergence on 
compact intervals.  
\end{proposition}
\begin{proof}
Let $\{\eta_n\}$ be  a sequence in $\cK$. By Proposition \ref{energy},
for a. e. $t\in\R$, $E(\dot \eta_n(t))=c$, and then $\{A_L(\eta_n)\}$ is bounded.
As in Proposition \ref{weak=fix} we obtain a sequence $n_k\to\infty$ such that
$\eta_{n_k}$ converges to $\eta:\R\to M$ uniformly on each $[a,b]$ and
then $\eta$ is static.
\end{proof}

\begin{proposition}\label{Acoincide}
Two dominated functions that coincide on 
$\cM=\bigcup\limits_{\eta\in\cK}\om(\eta)$ also coincide on $\cA$.
 \end{proposition}
 \begin{proof}
  Let $\fui_1$, $\fui_2$ be two dominated functions coinciding on
  $\cM$. Let $y\in\cA$ and $\eta\in\cK(y)$. Let $(t_n)_n$ be a
  diverging sequence such that $\lim_n\eta(t_n)= x\in\cM$. 
By (5) in Corollary \ref{aubry} 
\[\fui_i(y) = \fui_i(\eta(0)) - \Phi(y, \eta(0)) = \fui_i(\eta(t_n)) -
\Phi(y, \eta(t_n))\] 
for every $n\in N, i=1,2$. Sending $n$ to $\infty$, we get
\begin{align*}
\fui_1(y)&= \lim_{n\to\infty}\fui_1(\eta(t_n)) -\Phi(y, \eta(t_n))
=\fui_1(x) - \Phi(y, x) =\fui_2(x) - \Phi(y, x) \\
&= \lim_{n\to\infty}\fui_2(\eta(t_n)) -\Phi(y, \eta(t_n))
=\fui_2(y).
\end{align*}
%with $y$ an arbitrary point of $\cA$.
 \end{proof}
  \begin{proposition}\label{nonincreasing}
Let $\eta\in\cK$, $\psi\in C(M)$ and $\fui$ be a dominated function. 
Then the function
$t\mapsto(\cL_t\psi)(\eta(t))-\fui(\eta(t))$ is nonincreasing on $\R_+$.
 \end{proposition}
 \begin{proof} From (5) in Corollary \ref{aubry}, for $t<s$ we have 
   \[(\cL_s\psi)(\eta(s))-(\cL_t\psi)(\eta(t))\le\int_t^s L(\eta(\tau),\dot\eta(\tau))d\tau 
=\fui(\eta(s))-\fui(\eta(t))\]
 \end{proof}

 \begin{lemma}\label{perturbacion}
   There is $R>0$ such that, if $\eta$ is any curve in $\cK$ and $\lam$ 
is sufficiently close to 1, we have
   \begin{equation}\label{ineq}
\int_{t_1}^{t_2}L(\eta_\lam,\dot{\eta}_\lam)\le
\Phi(\eta_\lam(t_1),\eta_\lam(t_2))+R(t_2-t_1)(\lam-1)^2
   \end{equation}
for any $t_2>t_1$, where  $\eta_\lam(t)=\eta(\lam t)$.
\end{lemma}
\begin{proof}
Let $K=\max\{\|v\|:E(v)=c\}$,  $a=\sup\{\|L_{vv}(v)\|:\|v\|\le 2K\}$. 
For $\lam\in(1-\de,1+\de)$ fixed, using Proposition \ref{energy} 

\begin{align*}
\int_{t_1}^{t_2}L(\eta_\lam(t),\dot\eta_\lam(t))dt
&=\int_{t_1}^{t_2}[L(\eta(\lam t),\dot\eta(\lam t))+(\lam-1)
L_v(\eta(\lam t),\dot\eta(\lam t)) \dot\eta(\lam t)\\
&+\frac 12(\lam-1) ^2L_{vv}(\eta(\lam t),\mu\dot\eta(\lam t))(\dot\eta(\lam t))^2]\, dt\\
&\le\lam \int_{t_1}^{t_2}L(\eta(\lam t),\dot\eta(\lam t))\,dt+\frac 12 (t_2-t_1)aK^2(\lam-1)^2\\
&=\Phi(\eta(\lam t_1),\eta(\lam t_2))+\frac 12(t_2-t_1)aK^2(\lam-1)^2 
 \end{align*}
\end{proof}
\begin{proposition}\label{superdiff}
   Let $\eta\in\cK$, $\psi\in C(M)$ and $\fui$ be a dominated function. 
Assume that $D^+((\psi-\fui)\circ\eta)(0)\entre\{0\}\ne\emptyset$
where $D^+$ denote the super-differential.
Then for all $t>0$ we have
\begin{equation}\label{2}
(\cL_t\psi)(\eta(t))-\fui(\eta(t))<\psi(\eta(0))-\fui(\eta(0))
\end{equation}
 \end{proposition}
 \begin{proof}
   Fix $t>0$. By (5) in Corollary \ref{aubry} it is enough to prove
   \eqref{2} for $\fui=-\Phi(\cdot,\eta(t))$. Since $\cL_t(\psi+a)=\cL_t\psi+a$
we can assume that $\psi(\eta(0))=\fui(\eta(0))$.
\[(\cL_t\psi)(\eta(t))-\fui(\eta(t))=(\cL_t\psi)(\eta(t))\le
\int_{(1/\lam-1)t}^{t/\lam}L(\eta_\lam,\dot{\eta}_\lam)+\psi(\eta((1-\lam)t)),\]
thus, by Lemma \ref{perturbacion}
 \[(\cL_t\psi)(\eta(t))-\fui(\eta(t))\le
\psi(\eta((1-\lam)t))-\fui(\eta((1-\lam)t))+Mt(\lam-1)^2.\]
If $m\in D^+((\psi-\fui)\circ\eta)(0)\entre\{0\}$, we have
 \[(\cL_t\psi)(\eta(t))-\fui(\eta(t))\le m((1-\lam)t)+o((1-\lam)t))+Mt(\lam-1)^2,\]
where $\lim\limits_{\lam\to 1}\dfrac{o((1-\lam)t)}{1-\lam}=0$. 
Choosing appropriately $\lam$
close to $1$, we get
\[(\cL_t\psi)(\eta(t))-\fui(\eta(t))<0.\]
\end{proof}
\begin{proposition}\label{gonza}
  Let $u:M\to\R$ be a backward weak KAM solution, then
  \begin{align*}
    u(x)&=\min_{p\in\cA}u(p)+h(p,x).\\
          h(x,y)&=\min_{q\in\cA}h(x,q)+h(q,y)=\min_{q\in\cA}\Phi(x,q)+\Phi(q,x)
  \end{align*}
\end{proposition}
\begin{corollary}\label{kam}
  Let $C\subset M$ and $w_0:C\to\R$ be bounded from below. Let
\[w(x)=\inf_{z\in C}w_0(z)+\Phi(z,x)\]
\begin{enumerate}
\item\label{max-dom}
$w$ is the maximal dominated function not exceeding $w_0$ on $C$.
\item\label{aubry-kam}
If $C\subset\cA$, $w$ is a backward weak KAM solution.
\item\label{dom-coinc} If for all $x,y\in C$
\[w_0(y)-w_0(x)\le\Phi(x,y),\]
then $w$ coincides with $w_0$ on $C$.
\end{enumerate}
\end{corollary}

\subsection{Convergence of the Lax-Oleinik semigroup}
\label{sec:lax-converge}
We assume again that $M$ is compact. Adding $c(L)$ to 
$L$ we can assume that $c(L)=0$. The main result of this section is
Theorem \ref{conv-lax}.
\begin{proposition}\label{candidate}
  Let  $u\in C(M)$
  \begin{enumerate} [(a)]
  \item 
 Suppose $\psi=\lim\limits_{n\to\infty}\cL_{t_n}u$ for some $t_n\to\infty$, then
\begin{equation} \label{eq:limite}
  \psi\ge v(x):=\min_{z\in M}u(z)+h(z,x).
\end{equation}
\item  Suppose $\cL_tu$ converges as $t\to\infty$, then the limit is
function $v$ defined in \eqref{eq:limite}.
\end{enumerate}
\end{proposition}
\begin{proof}
(a)  For $x\in M$, $n\in\N$ let  $\ga_n\in\cWD([0,t_n])$ be such that $\ga(t_n)=x$ and
\begin{equation}\label{other}
\cL_{t_n}u(x)=u(\ga_n(0))+A_L(\ga_n).
\end{equation}
Passing to a subsequence if necessary we may assume that
$\ga_n(0)$ converges to $y\in M$. Taking $\liminf$ in \eqref{other},
we have from item (5) of Proposition \ref{prop-h}
\[\psi(x)= u(y)+\liminf_{n\to\infty}A_L(\ga_n) \ge u(y)+h(y,x). \]

(b)  For $x\in M$ let $z\in M$ be such that $v(z)=u(z)+h(z,x)$.
Since $\cL_t u(x)\le u(z)+h_t(z,x)$,  we have
\[\lim_{t\to\infty}\cL_t u(x)\le \liminf_{t\to\infty}u(z)+h_t(z,x)= v(z)\]
which together with item (a) gives $\lim\limits_{t\to\infty}\cL_t u=v$.
\end{proof}
%Thus, given $u\in C(M)$ our goal is to prove that
%$\cL_tu$ converges to $v$ defined in \eqref{eq:limite}.

Using Proposition \ref{gonza} we can write \eqref{eq:limite} as
\begin{align}\label{eq:limite1}
  v(x)&=\min_{y\in\cA}\Phi(y,x)+w(y)\\
  \label{eq:w}
  w(y)&:= \inf_{z\in M}u(z)+\Phi(z,y)
\end{align}
Item \eqref{max-dom} of Corollary \ref{kam} states that $w$
is the maximal dominated function not exceeding $u$.
Items \eqref{aubry-kam}, \eqref{dom-coinc}  of Corollary \ref{kam} imply that
$v$ is the unique backward weak KAM solution that coincides with $w$
on $\cA$.
\begin{proposition}\label{dom-conv}
 Suppose that $u$ is dominated, then  $\cL_tu$ converges uniformly as
$t\to\infty$ to the function $v$ given by \eqref{eq:limite}.
\end{proposition}
\begin{proof}
Since $u$ is dominated, the function $t\mapsto\cL_tu$ is nondecreasing.
As well, in this case, $w$ given by \eqref{eq:w} coincides with $u$.
Items \eqref{max-dom} and \eqref{dom-coinc} of
Corollary \ref{kam} imply that $v$ is the maximal dominated function
that coincides with $u$ on $\cA$ and then $u\le v$ on $M$.

Since the semigroup $\cL_t$ is monotone and $v$ is a backward weak KAM solution
\[\cL_tu\le\cL_tv=v\hbox{ for any } t>0.\]
Thus the uniform limit $\lim\limits_{t\to\infty}u$ exists.
\end{proof}
Henceforth we assume that there are $C_1, C_2>0$ such that
  \eqref{vlvbound} holds.
\begin{proposition}
 Let $\de>0$ and suppose $u:M\to\R$ is
  bounded, then the family $\{\cL_tu:t\ge\de\}$ is uniformly
  bounded and equicontinuous.
\end{proposition}
\begin{proof}
  By Proposition \ref{htlip} there is $K_\de>0$ such that
$\{h_t:t\ge\de\}$ is uniformly $K_\de$-Lipschitz.

  Let $x, y\in M$, $t\ge\de$. For $\ep>0$ there is $z_t\in M$ such that 
  \begin{equation}
    \label{eq:6}
    \cL_tu(y)>u(z_t)+h_t(z_t,y)-\ep
  \end{equation}
  Thus
  \[  \cL_tu(x)- \cL_tu(y)\le
  u(z_t)+h_t(z_t,x)-u(z_t)-h_t(z_t,y)+\ep\le K_\de d(x,y)+\ep, \]
implying that $\{\cL_tu:t\ge\de\}$ is equicontinuous.

Let $C_\de>0$ be given by Corollary \ref{wkam-cor}. From \eqref{eq:6} we get
\begin{align*}
-\|u\|_\infty -C_\de-&\ep\le u(z_t)+h_t(z_t,y)-\ep\\&<\cL_tu(y)
 \le u(z_t)+h_t(z_t,y)\le\|u\|_\infty+C_\de
\end{align*}
proving that $\{\cL_tu:t\ge\de\}$ is uniformly bounded.
\end{proof}
To prove the convergence of $\cL_t$ we will follow the lines in .

For $u\in C(M)$ let
\[\om_\cL(u):=\{\psi\in C(M):\exists t_n\to\infty \hbox{ such that } 
\psi=\lim_{n\to\infty}\cL_{t_n}u\}.\]
\begin{align}\label{uu}
  \underline{u}(x)&:=\sup\{\psi(x):\psi\in\om_\cL(u)\}\\
  \overline{u}(x)&:=\inf\{\psi(x):\psi\in\om_\cL(u)\}\label{ou}
\end{align}

\begin{corollary}\label{v<us}
  Let $u\in C(M)$, $v$ be the function given by \eqref{eq:limite},
  $\underline{u}, \overline{u}$ defined in \eqref{uu} and \eqref{ou}. Then 
  \begin{equation}\label{eq:v<u}
    v\le \overline{u}\le\underline{u}
  \end{equation}
\end{corollary}
  
\begin{proposition}
For $u\in C(M)$,  function $\underline{u}$
given by   \eqref{uu} is dominated.
\end{proposition}
\begin{proof}
Let $x,y\in M$. Given $\ep>0$ there is $\psi=\lim\limits_{n\to\infty}\cL_{t_n}u$ 
such that $\underline{u}(x)-\ep<\psi(x).$ For $n>N(\ep)$ and $a>0$
\[\underline{u}(x)-2\ep<\psi(x)-\ep\le \cL_{t_n}u(x)=\cL_a(\cL_{t_n-a}u)(x)\le
\cL_{t_n-a}u(y)+h_a(y,x).\]
Choose a divergent sequence  $n_j$ such that $(\cL_{t_{n_j}-a}u)_j$
converges uniformly. For $j>\bar N(\ep)$, $\cL_{t_{n_j}-a}u(y)<\underline{u}(y)+\ep$,
and then
\[\underline{u}(x)-3\ep<\cL_{t_{n_j}-a}u(y)+h_a(y,x)-\ep<\underline{u}(y) +h_a(y,x).\]
\end{proof}

\begin{proposition}\label{superincreasing}
  Suppose $\fui$ is dominated and $\psi\in\om_\cL(u)$. For any 
$y\in\cM$ there exists $\ga\in\cK(y)$ such that the function
$t\mapsto\psi(\ga(t))-\fui(\ga(t))$ is constant. 
\end{proposition}
\begin{proof}
 Let $(s_k)_k$ and $(t_k)_k$ be diverging sequences, $\eta$ be a
 curve in $\cK$ such that $y=\lim\limits_k \eta(s_k)$, and $\psi$ is
 the  uniform limit of $\cL_{t_k}u$. As in Proposition \ref{h-wkam}, 
we can assume that the sequence of functions $t\mapsto\eta(s_k+t )$ converges 
uniformly on compact intervals to $\ga:\R\to M$, and so  $\ga\in\cK$. 
We may assume moreover that $t_k-s_k\to\infty$, as $k\to\infty$, and
that $\cL_{t_k-s_k}u$ converges uniformly to $\psi_1\in\om_\cL(u)$. 
By (2) and (5) in Proposition \ref{prop-lax}
\[\|\cL_{t_k}u-\cL_{s_k}\psi_1\|_\infty\le\|\cL_{t_k-s_k}u -\psi_1\|_\infty\]
which implies that $\cL_{s_k}\psi_1$ converges uniformly to $\psi$. 
From Proposition \ref{nonincreasing}, we have that for any $\tau\in\R$
$s\mapsto(\cL_s\psi_1)(\eta(\tau+s))-\fui(\eta(\tau+s))$ is a nonincreasing 
function in $\R^+$, and hence it has a limit $l(\tau)$ as $s\to\infty$, 
which is finite since $l(\tau)\ge-\|\overline u-\fui\|_\infty$. 
Given $t>0$, we have
\[l(\tau) = \lim_{k\to\infty} (\cL_{s_k+t}\psi_1) (\eta(s_k +\tau+t))-\fui(\eta(s_k+\tau+t)) = (\cL_t\psi) (\ga(\tau+t))-\fui(\ga(\tau+t))\]
The function $t\mapsto(\cL_t\psi)(\ga(\tau+t))-\fui(\ga(\tau+t))$ is therefore constant on $\R^+$. 
Applying Proposition \ref{superdiff} to
the curve $\ga(\tau +\cdot)\in\cK$, we have
$D^+((\psi-\fui)\circ\ga)(\tau)\entre\{0\} =\emptyset$ for any
$\tau\in\R$. This implies that $\psi-\fui$ is constant on $\ga$.
\end{proof}
\begin{proposition}\label{acercanse}
  Let $\eta\in\cK$, $\psi\in\om_\cL(u)$ and $v$ be defined by \eqref{eq:limite}.
For any $\ep>0$ there exists $\tau\in\R$ such that 
\[\psi(\eta(\tau))-v(\eta(\tau))<\ep.\]
\end{proposition}
\begin{proof}
  Since the curve $\eta $ is contained in $\cA$, we have
\[v(\eta(0)) = \min_{z\in M} u(z) + \Phi(z,\eta(0)),\] 
and hence $v(\eta(0)) = u(z_0) + \Phi(z_0, \eta(0))$, for some
$z_0\in M$. Take a curve $\ga:[0,T]\to M$ such that
\[v(\eta(0))+\frac\ep 2 = u(z_0)+\Phi(z_0,\eta(0))+\frac\ep 2
> u(z_0)+\int_0^TL(\ga,\dga)\ge\cL_Tu(\eta(0)).\]
Choosing a divergent sequence $(t_n)_n$ such that $\cL_{t_n}u$
converges uniformly to $\psi$ we have for $n$ sufficiently large
\[\|\cL_{t_n}u-\psi\|_\infty<\frac\ep 2, \quad t_n-T>0.\]
Take $\tau=t_n-T$ 
\begin{align*}
  \psi(\eta(\tau))-\frac\ep 2&<\cL_{t_n}u(\eta(\tau))=\cL_{\tau}\cL_Tu(\eta(\tau))
=\cL_Tu(\eta(0))+\int_0^\tau L(\eta,\dot\eta)\\
&<\frac\ep 2+v(\eta(0))+\int_0^\tau L(\eta,\dot\eta)=\frac\ep 2+v(\eta(\tau))
\end{align*}

\end{proof}
From Propositions \ref{superincreasing} and \ref{acercanse} we obtain
\begin{theorem}\label{Mcoincide}
 Let $\psi\in\om_\cL(u)$ and $v$ be defined by \eqref{eq:limite}. Then $\psi=v$ 
on $\cM$.
\end{theorem}
\begin{theorem}\label{conv-lax}
  Assume that there are $C_1, C_2>0$ such that
  \eqref{vlvbound} holds.
  Let $u\in C(M)$, then $\cL_tu$ converges uniformly as $t\to\infty$ to $v$ 
given by \eqref{eq:limite}.
\end{theorem}
\begin{proof}
 The function $\underline{u}$
is dominated and coincides with $v$ on $\cM$ by Theorem 
\ref{Mcoincide}. Proposition \ref{Acoincide} implies that
$\underline{u}$ coincide with $v$ on $\cA$ and so does with $w$. 
By item \eqref{max-dom} of Corollary \ref{kam} we have  $\underline{u}\le v$.
\end{proof}

\section{Hamilton-Jacobi equation}
\label{sec:HJ}

\subsection{Viscosity solutions}

\begin{definition}
  Let $N$ be a manifold, $F:T^*N\to\R$ be continuous and consider the
 Hamilton-Jacobi equation
  \begin{equation}
  \label{HJ}
  F(u(z),du(z))=0
\end{equation}
We say that $u\in C(N)$ is a {\em viscosity subsolution} of \eqref{HJ}
if  $\forall\fui\in C^1(N)$ s.t. $z_0\in N$ is a point of local maximum
    of $u-\fui$ we have 
    \[F(u(x_0),d\fui(x_0)) \le 0.\]
We say that   $u\in C(N)$ is a {\em  viscosity supersolution} of  \eqref{HJ} if
$\forall\fui\in C^1(N)$ s.t. $z_0\in N$ is a point of local minimum of
$u-\fui$ we have
\[F(u(x_0),d\fui(z_0))\ge 0.\]
We say that $u$ is a {\em viscosity solution} if it is both, a
subsolution and a supersolution.
\end{definition}
Let $\pi^*:T^*M\to M$ be the natural projection and 
define the sub-Riemannian Hamiltonian $H: T^*M\to\R$ by
\[H(p)=\max\{p(v)-L(v):v\in\cD , \pi(v)=\pi^*(p)\}.\]
For $\om\in\Om^1(M)$, the Hamiltonian asociated to $L-\om$ is
$H(p+\om(\pi^*(p)))$
\begin{proposition}\label{solHJt}
 Let $u\in C(M)$, then
 $w:M\times[0,\infty[\to\R$ given by $w(x,t)=\cL_tu(x)$, is a viscosity solution of  
\[%begin{equation}  \label{eq:HJt}
  \begin{cases}
  w_t(x,t)+ H(d_xw(x,t))=0, & x\in M, t>0\\
 w(x,0)=u(x) &
\end{cases}
\]%end{equation}
\end{proposition}

\begin{proposition}\label{sub=domina}
    Let $\lam\ge 0$, then $u\in C(M)$ is a viscosity subsolution of
    \begin{equation}
  \label{eq:dHJ}
  \lam u(x)+H(du(x))-c=0,
\end{equation}
if and only if 
\begin{equation}
  \label{lam-dom}
    u(\ga(b))e^{\lam b}-u(\ga(a))e^{\lam a}\le A_{L+c,\lam}(\ga)
    \hbox{ for any }\ga\in\cWD([a,b])
  \end{equation}
\end{proposition}
%The proof is standard.
\begin{corollary}\quad
  \begin{enumerate}
  \item A viscosity subsolution of \eqref{eq:dHJ}
is Lipschitz for the metric $d$.
    
  \item Considering the case $\lam=0$
\begin{equation}
     \label {eq:HJ}
     H(du(x))-c=0,
   \end{equation}
\[c(L)=\inf\{c\in\R: \eqref{eq:HJ} \hbox{ has a viscosity subsolution } u\in C(M)\}\]
\end{enumerate}
\end{corollary}

 \begin{proposition}\label{value+time}\quad
   \begin{enumerate}
   \item For $\lam>0$,  let $u_\lam$  be the value function,
     then $u_\lam +c/\lam$ is a viscosity solution of \eqref{eq:dHJ}.
   \item    If $u=\cL_tu+ct$ for any $t>0$, then $u$ is a viscosity solution of
 \eqref{eq:HJ}.
   \end{enumerate}
   
 \end{proposition}
 %The proofs are standard and in fact
 Item (2) of Proposition \ref{value+time} can be considered a special case 
of Proposition \ref{solHJt} since $w(x,t)=\cL_tu(x)=u(x)-ct$ and then
$d_xw(x,t)=du(x)$, $w_t=-c$.

Under some mild assumption that can be fulfilled for a sub-Riemannian
Hamiltonian, the comparison principle for \eqref{eq:dHJ} holds for
$\lam>0$. In particular \eqref{eq:dHJ} has a unique viscosity solution.
An instance of such an assumption is
\begin{equation}
  \label{eq:bar}
 \Big|\frac{\partial H}{\partial x}(x,p)\Big|\le (1+|p|) \Phi(H(x,p)) 
\end{equation}
  for a continuous function $\Phi:\R\to\R^+$, see page 35 in
  \cite{Ba}.   This fact implies the following.

  \begin{corollary}
    Assume that \eqref{eq:bar} holds, 
    then there is only one constant $c$ such that \eqref{eq:HJ} has
    viscosity solutions.
  \end{corollary}
  \begin{proof}
    Suppose that there are $c_1<c_2$ such that \eqref{eq:HJ} admits
    viscosity solutions $u_1,u_2$ respectively. Adding a constant to $u_1$
    we may suppose $u_1> u_2$. Let $c_1<c<c_2$.
    For $\lam>0$ small enough $u_1$ and
    $u_2$ are a subsolution and a supersolution of
    \eqref{eq:dHJ}. By  comparison, $u_1\le u_2$ which is a contradiction.
  \end{proof}

  \section{Aubry-Mather theory}
\label{sec:aubry-mather}
In this section we extend Aubry-Mather  theory to our setting.
We take Ma\~n\'e's approach  using closed measures.
We assume again that $M$ is compact.

Let $g$ be a Riemannian metric on $M$ that tames the sub-Riemannian metric.
The values of $g$ outside $\cD$ will not play any role.
Let $C^0_\ell$ be the set of continuous functions 
 $f: TM\to\R$ having linear growth, i.e.
 $$
 \lV f\rV_\ell:=\sup_{v\in \cD}\frac{|f(v)|}{1+\lV v\rV_g}<+\infty.
 $$
 Let $\cP_\ell$  be the set of Borel probability measures on $TM$ such
 that
 $$
 \int_{TM}\lV v\rV_g\ d\mu<+\infty,
 $$
 endowed with the topology such that $\lim_n\mu_n=\mu$ if and only if
 \[%begin{equation}\label{defmulim}
 \lim_n\int_{TM} f\;d\mu_n=\int_{TM} f\;d\mu
 \]%end{equation}
 for all $f\in C^0_\ell$.
 
 Let $(C^0_\ell)'$ be the dual of $C^0_\ell$. Then  $\cP_\ell$ is
 naturally embedded in $(C^0_\ell)'$ and its topology coincides 
 with that  induced by the weak* topology on $(C^0_\ell)'$. 
 This topology is metrizable, indeed,
 let $\{f_n\}$ be a sequence of functions with compact support 
 on  $C^0_\ell$ which is dense on $C^0_\ell$ in the topology 
 of uniform convergence on compact sets of $TM$. 
 The metric $d_\ell$ on $\cP_\ell$ defined by
 \[   d_\ell(\mu_1,\mu_2)=
   \Big|\int_{TM} \|v\|_g \;d\mu_1-\int_{TM}\|v\|_g\;d\mu_2\Big|
   +\sum_n\frac{1}{2^n\|f_n\|_\infty}
   \Big|\int_{TM} f_n\; d\mu_1-\int_{TM} f_n\; d\mu_2\Big|\]
 gives the topology of $\cP_\ell$.

 Let $\cP_\cD$ be the set of measures in $\cP_\ell$ that are supported
 in $\cD$.
 \begin{proposition}\label{Ma}
   The function $A_L:\cP_\cD\to\R\cup\{+\infty\}$, defined as
   $A_L(\mu)= \int_{\cD }L\,d\mu$
 is lower semicontinuous and for $c\in\R$ the set 
\[A(c):=\{\mu\in \cP_\ell:\ A_L(\mu)\le c\}\] 
is compact in $\cP_\cD$.
\end{proposition}
\begin{definition}\label{vakonomic}
  A measure $\mu$ in $\cP_\cD$ is called {\em closed} or  {\em vakonomic}
(see \cite{G})  if
 \begin{equation}\label{eq:closed}
      \int_{TM}  D\varphi(\pi(v)) v\ d\mu(v)=0
\end{equation}
for all $\varphi\in C^1(M).$ We denote by  $\cVD$ the convex closed
set of vakonomic measures.
\end{definition}
For $\mu\in\cVD$ we define its rotation vector $\rho(\mu)\in H_1(M,\R)$ by
  \begin{equation}\label{eq:rotation}
    \lip [\om], \rho(\mu)\rip=\int_{TM} \om\ d\mu,\quad \om\in\Om^1(M), d\om=0
  \end{equation}
  For $\ga\in\cWD([a,b])$ we define the measure $\mu_\ga$  by 
\begin{equation}    \label{eq:example}
\int_{TM} f\ d\mu_\ga = \frac 1{b-a}\int_a^b f(\dga),\quad f\in C^0_\ell,
  \end{equation}
  when $\ga(a)=\ga(b)$, $\mu_\ga\in\cVD$  and it is called holonomic.
  We denote by $\cH_\cD$ the set of holonomic measures.

  More generally, if $\ga_n\in\cWD([a_n,b_n])$, $b_n-a_n\to\infty$ and
  $\mu=\lim\limits_n \mu_{\ga_n}$ exists then $\mu\in\cVD$. Indeed,
we observe again that if $\af\in\cWD$ is a minimizing geodesic with $\|\dot\af\|=1$
a.e. we have that $-C(0)\ell(\af)\le A_L(\af)\le A(1) \ell(\af)$. Fix $x_0\in M$,
take unit speed geodesics $\af_n$ from $x_0$ to $\ga_n(a_n)$ and $\be_n$
from $\ga_n(b_n)$ to $x_0$, and the concatenation $\de_n$ of 
$\af_n, \ga_n$ and $\be_n$. Then $\de_n$ is a closed curve and
for any $f\in C^0_\ell$ we have
\[\lim_n \int_{TM} f \ d\mu_{\de_n}=\lim_n \int_{TM} f\ d\mu_{\ga_n}=
\int_{TM} f \ d\mu.\]
  Thus $\mu\in\cVD$ and $\rho(\mu)=\lim\limits_n \rho(\mu_{\de_n})$.
  
 Set $\G:=\{h\in H_1(M,\Z): kh\ne 0\ \forall k\in\N\}$ and
consider the abelian cover $\Pi:\bM\to M$
whose group of deck transformations is $\G$.
The distribution $\cD$ lifts to a  bracket generating
distribution $\bar\cD\subset T\bM$.

Fix a basis $c_1,\ldots,c_k$ of $H^1(M,\R)$ and fix
closed 1-forms $\om_1,\ldots,\om_k$ in $M$ such that $\om_i$ has
cohomology class $c_i$ . 
Let $ G : \bM\to H_1(M,\R)$ be given by
\[\lip \sum_i a_ic_i, G(x)\rip = \int_{\bar x_0}^x\sum_i a_i\bar\om_i , \]
where $\bar x_0$ is a base point in $\bM$ and $\bar\om_i$ is the lift of
$\om_i$ to $\bM$. Since $\bar\om_i$ is exact, the integral does not
depend on the choice of the path from $\bar x_0$ to $x$. Notice that the
function $G$ depends on the choice of $\bar x_0$, on the choice of the
basis $\{c_i\}$ and of representatives $\om_i$ in $c_i$ . 
In general, if $x, y$ belong to the same fiber, then $G(x)- G(y)$ 
is the transformation in $\G$ carrying $y$ to
$x$, condition which uniquely identifies it,
conversely, any deck transformation admits a representation of this
type.

Let $\bga\in W^{1,1}_{\bar\cD}([a,b])$ and $\ga=\Pi\circ\bga$.
If $\om\in\Om^1(M)$, $d\om=0$, there
is $f\in C^1(M)$ such that
\[\int_{TM}\om\ d\mu_\ga=\frac 1{b-a}\lip[\om],G(\bga(b))-G(\bga(a))\rip
  +f((\ga(b))-f(\ga(a)).\]
If $T\in\G$, $a<b$, take 
$\bga\in\cW_{\bar\cD}([a,b])$ such that  $\bga(b))=T\bga(a)$, then 
$T=G(\bga(b))-G(\bga(a))$. If  $\ga=\Pi\circ\bga $ then  $\ga(a)=\ga(b)$
and $\rho(\mu_\ga)=T/(b-a)$.

If $\bga_n\in\cW_\cD([a_n,b_n])$, $\ga_n=\Pi\circ\bga_n$, $b_n-a_n\to\infty$ and
  $\mu=\lim\limits_n \mu_{\ga_n}$ exists then
\[h=\lim\limits_{n\to\infty}  \frac{G(\bga(b_n))-G(\bga(a_n))}{b_n-a_n}\]
exists and $\rho(\mu)=h$.

\begin{lemma}
    The map $\rho:\cVD\to H_1(M,\R)$ is surjective
      \end{lemma}
      \begin{proof}
Since $\cVD$ is convex and the map $\rho$ is affine, we have that
$\rho(\cVD)$ is convex, and we have proved that $\rho(\cVD)$
contains $\G$.
\end{proof}
\begin{lemma}\label{Mane}
  If $M$ is connected the set $\overline{\cH_\cD}$ is convex
\end{lemma}
\begin{proof}
  Let $\eta_1, \eta_2\in \overline{\cH_\cD}$, and let $\lambda_1$ and $\lambda_2$
in $[0,1]$ be such that $\lambda_1+\lambda_2=1$.
There are sequence $\ga^i_n\in\cC_{S^i_n}(x^i_n,x^i_n)$, $i=1,2$ such that
$\mu_{\ga^i_n}$ converges to $\eta_i$.
There are $m^1_n, m^2_n\in\N$ such that  $T^i_n=m^i_nS^i_n\to +\infty$
and $\dfrac{T^1_n}{T^2_n}$
converges to $\dfrac{\lam_1}{\lam_2}$, then  $\dfrac{T^i_n}{T^1_n+T^2_n}$
converges to $\lam_i$.

Taking a minimizing geodesic $\af_n\in\cC_{1}(x^1_n,x^2_n)$, 
define the curve $\ga_n\in\cC_{T^1_n+T^2_n+2}(x^1_n,x^1_n)$ by
\[\ga_n(t)=\begin{cases}
  \ga^1_n(t)& t\in[0,T^1_n]\\ \af_n(t-T^1_n) & t\in[T^1_n,T^1_n+1]\\
\ga^2_n(t-T^1_n-1) & t\in[T^1_n+1,T^1_n+T^2_n+1]\\
\af_n(T^1_n+T^2_n+2-t)& t\in[T^1_n+T^2_n+1,T^2_n+T^1_n+2].
\end{cases}\]
Then $\mu_{\ga_n}$ converges to $\lambda_1\eta_1+\lambda_2\eta_2$.
\end{proof}
\begin{theorem}\label{vako=holo}
        $\cVD=\overline{\cH_\cD}$
      \end{theorem}
      \begin{proof}
        As in \cite{B}, by Lemma \ref{Mane} and 
      \begin{proposition}\label{cc+}
        Let $C$ be a closed convex subset of $\cP_\ell$ and let
        \[C^+=\{f\in C^0_\ell:\int_{TM} f d\mu\ge 0\;\;\forall \mu\in C\}.\]
        Then  \[C=\{\mu\in\cP_\ell:\int_{TM} f d\mu\ge 0\;\; \forall f\in C^+\},\]
      \end{proposition}
    \noindent it is enough to prove that  $\overline{\cH_\cD}^+\subset\cVD^+$.
    The proof of this fact is the same as the proof of a similar fact in \cite{B}.
      \end{proof}
\begin{proposition}\label{aprox-mu}
  If $\mu\in\cVD$ then there is a sequence
  $\mu_{\eta_n}\in\cH_\cD$ converging to $\mu$ such that que $A_L(\mu_{\eta_n})$
converges to $A_L(\mu)$
\end{proposition}
\begin{proof}
   The proof is similar to the one of Proposition 2-3.3 in \cite {CI}.
 \end{proof}
\begin{proposition}\label{c=min-med}
  \[\min\{A_L(\mu):\mu\in\cVD\}=-c(L).\]
%and if $A_L(\mu)=-c(L)$, then $\sop\mu\subset\cA$.
\end{proposition}
\begin{proof}
  Set $c=c(L)$ and let $x\in M$, then there is a
  sequence  $\ga_n\in\cC_{t_n}(x,x)$ with  $t_n\to\infty$ such that 
 \[h(x,x)=\lim_{n\to\infty}A_L(\ga_n)+ct_n.\]
Let $\mu_0\in\cVD$ be the limit of a subsequence of $\mu_{\ga_n}$, then
 \[0=\lim_{n\to\infty}\frac1{t_n}A_L(\ga_n)+c=\lim_{n\to\infty}A_L(\mu_{\ga_n})+c\ge
   A_L(\mu_0)+c\ge\inf\{A_L(\mu):\mu\in\cVD\}+c.\]
 Let $\mu\in\cVD$, by Proposition \ref{aprox-mu}, there is a sequence
 $\ga_n\in\cC_{t_n}(x_n,x_n)$ 
such that $\dfrac1{t_n}A_L(\ga_n)=A_L(\mu_{\ga_n})$ converges to $-c(L)$.
Since $A_L(\ga_n)+ct_n\ge\Phi(x_n,x_n)=0$,
\[A_L(\mu)+c=\lim_{n\to\infty}\frac1{t_n}A_L(\ga_n)+c\ge 0.\]
Thus $\inf\{A_L(\mu):\mu\in\cVD\}+c\ge 0$ and then 
 \[A_L(\mu_0)=\inf\{A_L(\mu):\mu\in\cVD\}=-c\]
\end{proof}
We define the set of Mather measures  as $\fM(L)=\{\mu\in\cVD:A_L(\mu)=-c(L)\}$.
and the set of projected Mather measures $\fM_0(L)=\{\pi_\#\mu:\mu\in\fM(L)\}$.
\begin{example}
  Let $M$, $U:M\to\R$ and $L$ be as in example 1.
Since $\min L=-\max U=-c(L)$, the support of any $\mu\in\fM(L)$ is
contained in $\{0\}\times U^{-1}(\max U)=\{0\}\times\cA$.

 If $\mu$ is a probability measure in $TM$ supported in
 $\{0\}\times\cA$, then $\mu\in\cVD$ and
$\int L d\mu=-\max U$ and thus $\mu\in\fM(L)$. Therefore $\fM_0(L)$
is the set of probabilty measures in $M$ supported in $U^{-1}(\max U)=\cA$.
\end{example}
\begin{definition}
  We define the effective Lagrangian $\mL:H_1(M,\R)\to\R$, or
 Mather $\be$ function, by
  \begin{equation}\label{eq:effL}
    \mL(h)=\inf\{A_L(\mu):\rho(\mu)=h\}
  \end{equation}
  It is clear that $\mL$ is convex and its Legendre transform,
  called effective Hamiltonian, $\mH:H^1(M,\R)\to\R$ is given by
  \begin{equation}\label{eq:effH}
    \mH([\om])=-\inf\{A_{L-\om}(\mu):\mu\in\cVD\}=c(L-\om)
  \end{equation}
\end{definition}
\section{Convergence of the discounted value function}
\label{sec:conv-disc-value}
%In this section we extend the result in \cite{DFIZ} holds to our setting.
Suppose $M$ is compact and let $c=c(L)$.
For $\lam>0$ let $u_\lam$ be the discounted value function.
\begin{proposition}
  For every $\mu\in\fM_0(L)$ we have $\int_M u_\lam\ d\mu\le - \frac c\lam$.
\end{proposition}
\begin{proof}
  Let  $\ga_k\in\cC_{t_k}(x_k,x_k)$. By item (2) of Proposition \ref{value-prop}
   \[u_\lam(\ga_k(t))\le e^{-\lam t}(u_\lam(x_k)+\int_0^t e^{\lam s}L(\dga_k(s))\ ds).\]
Integration by parts gives
  \begin{align*}
    \int_0^{t_k}u_\lam(\ga_k(t))\ dt &\le \frac{1-e^{-\lam t_k}}{\lam}u_\lam(x_k)+
 \int_0^{t_k}e^{-\lam t} \int_0^t e^{\lam s}L(\dga_k(s))\ ds\  dt\\
&=\frac{1-e^{-\lam t_k}}{\lam}u_\lam(x_k)-\frac{e^{-\lam t_k}}{\lam}\int_0^{t_k} e^{\lam s}L(\dga_k(s))\ ds
+ \frac 1\lam \int_0^{t_k} L(\dga_k)\\ &\le \frac {A_L(\ga_k)}\lam.
  \end{align*}
  If the sequence $\mu_{\ga_k}\in\cH_\cD$ converges to $\nu\in\cV_\cD$ we have 
 \[\int_M u_\lam\ d\pi_\#\nu\le \frac{A_L(\nu)}\lam.\]
In particular, for $\nu\in\fM(L)$ choose a sequenece $\mu_{\ga_k}\in\cH_\cD$ converging to $\nu$,
and then \[\int_M u_\lam\ d\pi_\#\nu\le -\frac c\lam.\]
\end{proof}
\begin{proposition}
 $u_\lam+\dfrac c\lam$ is uniformly bounded.
\end{proposition}
\begin{proof}
  Fix $z\in\cA$ and let $v(x)=h(z,x)$. Let $x\in M$.

By Proposition \ref{min-inf} there is $\ga_{\lam,x}\in\cWD(]-\infty,0])$
such that $\ga_{\lam,x}(0)=x$ and 
\begin{align*}
  u_\lam(x)&=\int_{-\infty}^0e^{\lam t}L(\dga_{\lam,x}(t))\ dt
 \ge \int_{-\infty}^0e^{\lam t}((v\circ\ga_{\lam,x})'-c)\ dt\\
&\ge v(x)-\int_{-\infty}^0e^{\lam t}(\lam v\circ\ga_{\lam,x}-c)\ dt
\ge v(x)-\max v-\frac c\lam.
\end{align*}
There is a sequence  $\ga_n\in\cC_{t_n}(x,x)$ with
$t_n\to\infty$ such that  
\[h(x,x)=\lim_{n\to\infty}A_L(\ga_n)+ct_n.\]
By \eqref{lam-dom}, $u_\lam(x)(e^{\lam t_n}-1)\le A_{L,\lam}(\ga_n)$.
Considering the function $a_n(t)=A_L(\ga_n|_{[0,t]})$ and integrating
by parts
\begin{align*}
 u_\lam(x)(e^{\lam t_n}-1)&\le A_L(\ga_n)e^{\lam t_n}-\int_0^{t_n}\lam e^{\lam t} a_n(t)\ dt\\
&\le A_L(\ga_n)e^{\lam t_n}  +\int_0^{t_n}\lam  e^{\lam t}(v(x)-v(\ga_n(t))+ct)\ dt\\
     &\le A_L(\ga_n)e^{\lam t_n}  +(e^{\lam t_n}-1)(v(x)-\min v)
       +c(e^{\lam t_n}t_n-\frac{e^{\lam t_n}-1}\lam)\\
 &= (e^{\lam t_n}-1)(A_L(\ga_n)+ct_n +v(x)-\min v-\frac c\lam)+A_L(\ga_n)+ct_n 
\end{align*}
dividing by $e^{\lam t_n}-1$ and letting $n\to\infty$,
\[u_\lam(x)+\frac c\lam\le h(x,x)+v(x)-\min v\le h(x,z)+2v(x)-\min v.\]
\end{proof}
\begin{corollary}
  If $u_{\lam_k}+\dfrac c{\lam_k}$ converges uniformly to $u$ for $\lam_k\to 0$, then
  $u$ is a weak KAM solution and $\int\limits_Mu\ d\mu\le 0$ 
\end{corollary}
\begin{proof}
The first assertion follows as in Theorem \ref{wkam-teo}.
We already proved that   $\lam_k \int\limits_M u_{\lam_k}d\mu\le -c$
for  $\mu\in\fM_0$. Thus
  $\int\limits_M (u_{\lam_k}+ \dfrac c{\lam_k})d\mu\le 0$ and
  letting $k\to\infty$ we get $\int\limits_Mu\ d\mu\le 0$.
\end{proof}
\begin{theorem}\label{conv-disc}
 Let $u_\lam$ be the discounted value function, then $u_\lam + \dfrac c\lam$
 converges uniformly as $\lam\to 0$  to the weak KAM solution 
$u_0$ given by either of the following ways
\begin{enumerate}[(i)]
\item it is the largest viscosity subsolution of \eqref{eq:HJ}  such
  that $\int\limits_M u\ d\mu\le 0$ for any $\mu\in\fM_0(L)$.
\item $u_0(x)=\inf\{\int\limits_Mh(y,x)\ d\mu(y): \mu\in\fM_0(L)\}$.
\end{enumerate}
\end{theorem}
Now the proof of Theorem \ref{conv-disc} goes as in \cite{DFIZ}

\section{Homogenization}
\label{sec:homog}
As in section \ref{sec:aubry-mather}, 
let $M$ be a compact manifold. We follow the approach of \cite{CIS}
and refer to that paper for the definition of the convergence of a family
of metric spaces $(\cM_n,d_n)$ to a metric space $(\cM,d)$ trough
continuous maps $F_n:\cM_n\to\cM$.
For the Carnot Caratheodory metric $d$ on $\bM$ for the
distribution $\bar\cD$ we have as in \cite{CIS}

\begin{proposition}
\[\lim_{\ep\to 0} (\bM , \ep d , \ep G ) = H_1 ( M , \R ) . \] 
\end{proposition}
Let $\bL:\bar\cD\to\R$ be the lift of $L$ and
 \[\bH(p)=\max\{p(v)-\bL(v):\bar\pi(v)=\bar\pi^*(p)\}.\]
We are interested in the limit as $\ep\to 0$ of the Cauchy problem
\begin{equation}\label{eq:cauchy}
  \begin{cases}
    w_t + \bH(dw/\ep) = 0,&  x\in\bM, t > 0; \\
    w(x,0)= f_\ep(x). &
  \end{cases}
  \end{equation}
  for $f_\ep : \bM \to\R$ continuous, uniformly $\ep G$–converging to
  some $f: H_1(M,\R)\to \R$.

  Since
  \[\bH( p/\ep)=\max\{p(v)-\bL(\ep v) :\bar\pi(v)=\bar\pi^*(p)\},\]
letting $\cL^\ep_t$ be the Lax-Oleinik semi-group for the Lagragian $\bL(\ep v)$,
by Proposition \ref{solHJt} we have that $u_\ep$ given by
$u_\ep(x,t)=\cL^\ep_tf_\ep(x)$  is a viscosity solution to \eqref{eq:cauchy}.
Under assumption \eqref{eq:bar} the viscosity solution is unique. We have
\begin{align}\nonumber
  u_\ep(x,t)&=\inf\{f_\ep(\bga(0))+\int_0^t\bL(\ep\dot{\bga}):\ga\in\cW_{\bar\cD}([0,t])\}\\\nonumber
&=\inf\{f_\ep(\af(0))+\ep A_{\bL}(\bar\af):\bar\af\in\cW_{\bar\cD}([0,t/\ep])\}\\\label{hbar}
&=\inf\{f_\ep(y)+\ep\bar h_{t/\ep}(y,x)\}
\end{align}
where $\bar h_\tau:\bM\times\bM\to\R$ is the minimal action for the
Lagrangian $\bL$.
\begin{theorem}\label{homo}
 Let $f_\ep : \bM \to\R$, $f: H_1(M,\R)\to \R$ be continuous, with $f$
 having at most linear growth, such that
$f_\ep$ uniformly $\ep G$–converges to $f$.
Then the family of functions $u_\ep : \bM\times [0,
+\infty[\to\R$ locally uniformly $\ep G$-converges in
$M\times]0, +\infty[$ to the viscosity solution $u : H_1(M, \R)\times [0, +\infty[\to R$ of the problem
\begin{equation}\label{eq:l-cauchy}
  \begin{cases}
u_t + \mH(du) = 0, & h \in H_1(M,\R), t > 0; \\
u(h,0)= f(h)&
\end{cases}
\end{equation}
\end{theorem}
We need the following Lemma
\begin{lemma}[\cite{CIS}]\label{ybound}
  Given a compact subset $K$ in $H^1(M, \R)$, and a compact interval
  $I\subset [0, +\infty]$,
there is a constant $C>0$ such that
$\ep d(x, y)\le C$
for $\ep > 0$ suitably small $\ep G(x)\in K$, $t\in I$ and $y\in\bM$
realizing the minimum in the \eqref{hbar}.
\end{lemma}
\begin{proof}[Proof of Theorem \ref{homo} ]
 The solution for the limit problem \eqref{eq:l-cauchy} is
\[u(z,t)=\inf\{f(q)+t\mL\Bigl(\frac{z-q}t\Bigr):q\in H_1(M,\R)\}.\]

Let $h\in H_1(M,\R)$, $t>0$. Let $\{\ep_n\}$ be a sequence converging
to zero and let $\{x_n\}, \{t_n\}$ be
subsequences in $\bM, (0,+\infty)$ such that ${\ep_n}G(x_n)$ converges
to $z$ and $t_n$ converges to $t$. Our goal is to prove that $u_{\ep_n}(x_n,t_n)$ 
converges to $u(z, t)$.

Let $y_n\in\bM$, and $\bga_n\in\cC_{t_n/\ep_n}(y_n,x_n)$ be such that
\[u_{\ep_n}(x_n,t_n)=f_{\ep_n}(y_n) +\ep_n\bar h_{t_n/\ep_n}(y_n,x_n)=
f_{\ep_n}(y_n) +\ep_n A_{\bL} (\bga_n)\]
  
By Lemma \ref{ybound} the sequence $\ep_nG(y_n)$ is bounded, so it
converges converges up to subsequences, to $q\in H_1(M,\R)$ and then
\[\lim_n \frac{\ep_n(G(x_n)-G(y_n))}{t_n}=\frac{z-q}t.\]
Consider the measure $\mu_{\ga_n} $ for $\ga_n=\Pi\circ\bga_n$
\[\bar h_{t_n/\ep_n}(x_n,t_n)=A_L(\ga_n)=t_n A_L(\mu_{\ga_n})/{\ep_n}.\]
Passing to a further subsquence we can assume that $\mu_{\ga_n}$
converges to $\mu$, then we have
\[\rho(\mu)=\dfrac{z-q}t\ \hbox{ and }\ \liminf_n A_L(\mu_{\ga_n})\ge A_L(\mu).\]
Thus
\[\liminf_n\ep_n \bar h_{t_n/\ep_n}(x_n,x_n)=\liminf_n t_n  A_L(\mu_{\ga_n})
 \ge tA_L(\mu) \ge t\mL\Bigl(\frac{z-q}t\Bigr).\]
By the the uniform $\ep_n G$-convergence of $f_{\ep_n}$ to $f$, we obtain
\[\liminf_n u_{\ep_n}(x_n,t_n)=
  \liminf_n f_{\ep_n}(y_n) +\ep\bar h_{t_n/\ep_n}(y_n,x_n)
\ge f(q)+t\mL\Bigl(\frac{z-q}t\Bigr)\ge u(z,t)\]
Let $l=\limsup_n u_{\ep_n}(x_n,t_n)$ and take a subsequence of $\{\ep_n\}$,
still denoted $\{\ep_n\}$, such that $l=\lim\limits_{n\to\infty}
u_{\ep_n}(x_n,t_n)$.
Choose $q\in H_1(M,R)$ such that $u(z,t)=f(q)+t\mL\Bigl(\dfrac{z-q}t\Bigr)$.
Let $\mu\in\cV_\cD$ be such that $\rho(\mu)=\dfrac{z-q}t$ and
$A_L(\mu)=\mL\Bigl(\dfrac{z-q}t\Bigr)$.
There is a sequence $\mu_{\eta_k}\in\cH_\cD$ converging to $\mu$ such that 
$\{A_L(\mu_{\eta_k})\}$ converges to $A_L(\mu)$.

We have that $\eta_k\in\cC_{T_k}(p_k,p_k)$ and we can assume that $T_k\ge 1$.
Take a subsequence $\{\ep_{n_k}\}$ such that $\ep_{n_k}T_k$ converges
to $0$.

Choose a lift $y_k$ of $p_{n_k}$ such that $d^*_k=d(y_k, x_{n_k})\le\diam M$.
Let $m_k\in\N$ be such that
\[(1+m_k)T_k\le\frac{t_{n_k}}{\ep_{n_k}}<(2+m_k)T_k,\]
then $\ep_{n_k}m_kT_k$ converges to $t$.

Let $\be_k\in\cW_{\bar\cD}([0,m_kT_nk])$ be the lift of the curve
$\eta_k$ travelled $m_k$ times such that $\be_k(m_kT_k)=y_k$, then 
$\dfrac{G(y_k)-G(\be_k(0))}{m_kT_k}=\rho(\mu_{\eta_k})$.
Thus $\ep_{n_k}(G(y_k)-G(\be_k(0)))=\ep_{n_k}m_kT_k \rho(\mu_{\eta_k})$
converges to $z-q$ and since $\ep_{n_k}(G(x_{n_k})-G(y_k))$ converges to
$0$, we have that $\ep_{n_k}G(\be_k(0))$ converges to $q$.

Take a minimizing geodesic $\be_k\in\cC_{0, d^*_k}(y_k, x_{n_k})$ with
$\|\dot\be_k\|=1$ a.e. Letting
$\af_k(s)=\be_k(d^*_ks/({\frac{t_{n_k}}{\ep_{n_k}}-m_kT_k}))$ we have that
$|\dot\af_k|\le d^*_k$ and then $A_{\bL}(\af_k)\le 2 A(\diam M) T_k$.
\begin{align*}
l&=\lim_k u_{\ep_{n_k}}(x_{n_k},t_{n_k})\le \limsup_k f_{\ep_{n_k}}(\be_k(0)) +\ep_{n_k} (A_{\bL}(\be_k)+A_{\bL}(\af_k)))\\
&\le \lim_k f_{\ep_{n_k}}(\be_k(0))+\ep_{n_k}m_kT_kA_L(\mu_{\eta_k})+2\ep_{n_k}T_k A(\diam M)
=f(q)+t\mL\Bigl(\frac{z-q}t\Bigr).
\end{align*}
\end{proof}
\section{Open questions}
\begin{itemize}
\item Is any viscosity solution of \eqref{eq:HJ} a weak KAM solution?
\item Is the support of any $\mu\in\fM_0(L)$ contained in $\cA$?
\end{itemize}

\end{document}